\documentclass[11pt]{amsart}

\usepackage{amsmath}
\usepackage[font=small]{caption}
\usepackage{amsthm}
\usepackage{amsfonts}
\usepackage[pdftex]{graphicx}
\usepackage[margin=.5cm, format=hang]{subfig}
\usepackage[alphabetic]{amsrefs}
\usepackage{amssymb}                                                                                                                                                                                                                                                                                                                                                                                                                                                                                                                                                                                                                                                                                                                                                                                                 
\usepackage{pinlabel}
\usepackage{morefloats}
\usepackage{enumerate}
\usepackage{fullpage}
\usepackage{array}
\usepackage{multirow}
\usepackage{sidecap}
\usepackage{hhline}
\usepackage{multirow}
\usepackage{blkarray}

\theoremstyle{plain}
 \newtheorem{thm}{Theorem}[section]
\newtheorem{cor}[thm]{Corollary}
\newtheorem{prop}[thm]{Proposition}
\newtheorem{conj}[thm]{Conjecture}
\newtheorem{lem}[thm]{Lemma}
\theoremstyle{definition} \newtheorem{df}[thm]{Definition}
\theoremstyle{remark} \newtheorem{rmk}[thm]{Remark}
\theoremstyle{remark} 
\theoremstyle{remark} \newtheorem{ex}[thm]{Example}

\newcommand{\Q}{\mathbb{Q}}
\newcommand{\Z}{\mathbb{Z}}

\newcommand{\OS}{Ozsv\'ath and Szab\'o }

\newcommand{\cp}{\mathbb{C}P(2)}
\newcommand{\mubar}{\overline{\mu}}
\newcommand{\PZ}{\mathcal{P}_0}

\newcommand{\tld}[1]{\widetilde{#1}}

\newenvironment{enum}{
\begin{enumerate}
  \setlength{\itemsep}{1pt}
  \setlength{\parskip}{0pt}
  \setlength{\parsep}{0pt}
}{\end{enumerate}}

\title[]{Positive Links}

\author{Tim D. Cochran$^{\dag}$}
\address{Department of Mathematics MS-136, P.O. Box 1892, Rice University, Houston, TX 77251-1892}
\email{cochran@rice.edu}

\author{Eamonn Tweedy}
\address{Department of Mathematics MS-136, P.O. Box 1892, Rice University, Houston, TX 77251-1892}
\email{eamonn@rice.edu}

\thanks{$^{\dag}$Partially supported by the National Science Foundation  DMS-1006908}

\subjclass[2000]{Primary Secondary}

\begin{document}
\date{\today}
\begin{abstract}
Given a link $L \subset S^3$, we ask whether the components of $L$ bound disjoint, nullhomologous disks properly embedded in a simply-connected positive-definite smooth 4-manifold; the knot case has been studied extensively in \cite{CHH}. Such a $4$-manifold is necessarily homeomorphic to a (punctured) $\#_k\cp$. We characterize all links that are slice in a (punctured) $\#_k\cp$ in terms of ribbon moves and an operation which we call adding a generalized positive crossing. We find obstructions in the form of the Levine-Tristram signature function, the signs of the first author's generalized Sato-Levine invariants \cite{C4},  and certain Milnor's invariants. We show that the signs of coefficients of the Conway polynomial obstruct a $2$-component link from being slice in a single punctured $\cp$ and conjecture these are obstructions in general. These results have applications to the question of when a $3$-manifold bounds a $4$-manifold whose intersection form is that of some $\#_k\cp$. For example, we show that any homology $3$-sphere is cobordant, via a smooth positive definite manifold, to a connected sum of surgeries on knots in $S^3$.
\end{abstract}

\maketitle

\section{Introduction}\label{sec:intro}
There has been significant interest in studying cobordisms between closed, oriented, connected 3-manifolds.  In particular, given some non-zero subring $R$ of $\mathbb{Q}$ (e.g. $R = \mathbb{Q}$ or $R=\mathbb{Z}$), which 3-manifolds are $R$-homology cobordant?  This question is closely related to the study of knot concordance - in particular, when two knots $K_0,K_1 \subset S^3$ are concordant in $S^3 \times I$, then the manifolds obtained by $(p/q)$-framed surgery on $K_0$ and $K_1$ are $\mathbb{Z}$-homology cobordant.  Moreover, the n-fold cyclic branched covers of $K_0$ and $K_1$ are $\mathbb{Q}$-homology cobordant for any prime-power integer $n$.

A natural generalization of the equivalence relation of homology cobordism is to a sort of inequality.  In \cite{CGompf}, Cochran and Gompf defined: for two 3-manifolds, $M_1 \geq M_0$ if there is a smooth positive-definite cobordism $W$ from $M_0$ to $M_1$, i.e. a smooth 4-manifold $W$ with positive-definite intersection form and with $\partial W = (-M_0) \sqcup (M_1)$.  They also defined a relation ``$\geq$'' on knots in $S^3$, where $K_1 \geq K_0$ if there is a nullhomologous annulus $A$ properly embedded in a positive-definite cobordism from $S^3$ to $S^3$ such that $\partial A = (-K_0) \sqcup (K_1)$.  If $K_1 \geq K_0$, they show (by performing $(p/q)$-surgery on the concordance annulus) that $M_{K_1}(p/q) \geq M_{K_0}(p/q)$, where $M_{K_i}(p/q)$ denotes the $(p/q)$-framed surgery on $K_i \subset S^3$. In \cite{CHH}, Cochran, Harvey, and Horn generalize Cochran and Gompf's knot inequality via a family of relations $\left\{ \geq_n \right\}_{\mathbb{Z}_{\geq 0}}$ for knots in $S^3$.
% Note that their definition of $\geq_n$ further requires that the four-manifold $W$ be simply-connected.

These relations inspire the question: which $3$-manifolds can be related via a cobordism with \emph{unimodular} positive-definite intersection form?  Notice in particular that if $K_1 \geq_n K_0$ for any $n$, then there is a unimodular positive-definite cobordism from $M_{K_0}(p/q)$ to $M_{K_1}(p/q)$.  This motivates the question: Which knots are ``$n$-positive'', i.e. $\geq_n$ the unknot?  The inequality $\geq_n$ descends to the smooth knot concordance group $\mathcal{C}$, and one main thrust of \cite{CHH} is to study the filtration on $\mathcal{C}$ provided by the subsets $\mathcal{P}_n$ of $n$-positive concordance classes (for various $n$).

In \cite{CHH} it is shown that the signs of several knot concordance invariants obstruct membership in $\mathcal{P}_0$ - including \OS\!\!'s $\tau$ arising in knot Floer homology \cite{OzSz2}, Rasmussen's $s$ arising in Khovanov homology \cite{Ras1}, and the ``correction term'' $d$ for $(\pm1)$-surgery on $K$ \cite{OzSz:Tri},\cite{Pet}. 

 In the present paper, we study  $0$-positivity for concordance classes of links (see Definition \ref{def:ZP}) - a natural extension, as many 3-manifolds can't be constructed by performing surgery on a knot in $S^3$. One  says that $[L] \in \PZ$ if $L \subset S^3$ is slice in a positive-definite simply-connected four-manifold.  This manifold is called a $0$-\textbf{positon} for $L$.  We abuse notation  by writing ``$L \in \PZ$'' .

If $L \in \PZ$, then each of its components is a $0$-positive knot.  Since the question of which knots lie in $\PZ$ has been treated elsewhere, we seek to focus on aspects of linking modulo the knot type of the components. In this regard it is natural to ask when  $L \geq_0 L'$ for some link $L'$ that is \textbf{totally split} (i.e. one can find some pairwise disjoint open three-balls in $S^3$ each containing exactly one component of $L'$).  If  each of the components of the totally split link $L'$ were a $0$-positive knot, it would follow that $L$ is a $0$-positive link.  We show that:
{
\renewcommand{\thethm}{\ref{prop:fusion}}
\begin{prop}
A fusion of a boundary link is $\geq_0$ a totally split link.
\end{prop}
\addtocounter{thm}{-1}
}

This proposition has a rather interesting consequence:
{
\renewcommand{\thethm}{\ref{cor:fusion}}
\begin{cor}
Let $M^3$ be a homology sphere.  Then $M$ is cobordant to a connected sum of homology spheres $M^3_1, M^3_2, \ldots, M^3_k$ via a positive-definite unimodular cobordism $W^4$, with $H_1(W)=0$, where each $M_i$ is $(\pm 1)$-surgery on a knot $K_i \subset S^3$.
\end{cor}
\addtocounter{thm}{-1}
}

Recall that the class of links that are concordant to fusions of boundary links is conjectured to be equal to the class of links with vanishing Milnor's invariants ~\cite[Cor. 2.5]{C2}\cite[Question 16, p.66]{C4}\cite[p.572]{Le1}.  This suggest that Milnor's invariants might obstruct being $\geq_0$. We investigate this in Section~\ref{sec:milnor}.  Recall that the simplest Milnor's invariants are $\mubar_L(ij)$, $\mubar_L(ijk)$ and $\mubar_L(iijj)$ which can be identified with, respectively, the linking numbers, the triple linking numbers, and the Sat-Levine invariants. Further recall that in \cite{C4}, the first author defined a sequence $\left\{ \beta_n(L) \right\}$ of link concordance invariants for a two-component link $L$ - these generalize the Sato-Levine invariant $\beta(L) = \beta_1(L)$.  The first non-zero invariant $\beta_n(L)$ in the sequence coincides with the Milnor's invariant $\mubar_L(1111111...22)$.  We prove the following:
%and the following obstruction follows immediately from Theorem \ref{thm:betan} below:

\begin{prop}\label{prop:linking}
Let $L \subset S^3$ be a link which is $\geq_0$ a totally split link.  Then
\begin{enumerate}
\item For each $i,j$, $\mubar_L(ij) =0$ (Lemma \ref{lem:linkingnumber}).
\item For each $i, j, k$, $\mubar_L(ijk) = 0$ (Lemma \ref{lem:triplelinking}, this is first due to Otto \cite{Otto1}).
\item For each $i \neq j$, the first non-vanishing $\beta_n(L_{ij})$ for the two-component sublink $L_{ij}$ is non-positive (Corollary \ref{cor:betan}).
\end{enumerate}
\end{prop}

Item $(3)$ is significant because it indicates that Milnor's invariants of arbitrary length may obstruct membership in $\PZ$. On the other hand, we give an example of a link in $\PZ$ for which $\mubar(1234)$ is positive and an example for which $\mubar(1234)$  is negative. This indicates that precisely \textit{which} Milnor's invariants obstruct membership is subtle. This also shows that a link in $\PZ$ need not be null-link-homotopic.

Proposition~\ref{prop:fusion} might seem to suggest that if all of the components of $L$ are $0$-positive knots (e.g. if all components are slice knots), then Milnor's invariants may provide a \textit{complete} obstruction to $0$-positivity of $L$.  However, this is not the case, since the process used to ``separate'' the components of $L$ in the proof of Proposition \ref{prop:fusion} may lead to a totally split link with very complicated components; in a sense, one trades linking for knotting.  Indeed,  let $L$ be a two-component link obtained by negative-Whitehead-doubling each component of either of the Hopf links (see Example \ref{ex:sig}).  This link has unknotted components and is a boundary link (and thus $\geq_0$ a split link).  However, the classical link signature of $L$ is positive, and thus the following, a generalization of the classical work of Murasugi and Tristram on links, and of ~\cite[Prop. 4.1]{CHH} on knots, implies that $L \notin \PZ$:

{
\renewcommand{\thethm}{\ref{Theorem:signatures}}
\begin{thm}
 If $L \in\PZ$,  then the Levine-Tristram signature function of $L$ is non-positive.
\end{thm}
\addtocounter{thm}{-1}
}

Although one may obstruct membership in $\PZ$ using many concordance invariants, it remains to completely characterize these concordance classes.  Notice that if $V$ is a $0$-positon for $L$, then $V$ is a smooth manifold that is homeomorphic to a punctured connected sum of several copies of $\cp$.  Although it is not known at this time whether $\#_j \cp$ has a unique smooth structure, we proceed to study the set $\tld{\PZ}$ of concordance classes of links $L$ such that $L$ is slice in a (non-exotic) $\#_j \cp$. In Section \ref{sec:fam}  we give two characterizations of this set. We define a particular family of links called \textbf{null generalized Hopf links }and a closely related operation called \textbf{adding a generalized positive crossing} (these notions promise to be of independent interest). Then  we show:

{
\renewcommand{\thethm}{\ref{thm:ZP}}
\begin{thm}  Every concordance class in $\tld{\PZ}$ contains a representative which is a fusion of null generalized Hopf links; and contains a representative which is obtained from a ribbon link by adding generalized positive crossings.
\end{thm}
\addtocounter{thm}{-1}
}

Recall that it is known that the first non-vanishing coefficient of the Conway polynomial of a link is a concordance invariant ~\cite{C5}. This suggests another possible source of obstructions, especially since it is known that such coefficients are equal to a sum of Milnor's invariants ~\cite{C5,Le11}. More specifically, recall that when $L = (L_1,L_2)$ is a two-component link, the coefficient of $z$ in the Conway polynomial $\nabla_L(Z) = z(a_0 + a_1z^2 + a_2 z^4 + \ldots)$ of $L$ is equal to $-\ell k(L_1,L_2)$.  When $a_0$ vanishes, then $a_1 = -\beta(L)$, where the Sato-Levine invariant $\beta(L)$ is non-positive when $L \in \PZ$ by Proposition~\ref{prop:linking}. Hence one should ask whether there is a general rule governing the sign of the smallest degree non-vanishing coefficient of $\nabla_L(z)$ when $L \in \PZ$ (or $L \in \tld{\PZ}$).    We give evidence for this by proving the following:

{
\renewcommand{\thethm}{\ref{thm:conway}}
\begin{thm}
Let $L \subset S^3$ be a 2-component link which is slice in a punctured $\mathbb{C}P^2$, and suppose that the Conway polynomial $\nabla_L$ of $L$ is of the form
$$ \nabla_L(z) = z \left( a_k z^{2k} + a_{(k+1)} z^{2k+2} + \ldots + a_n z^{2n} \right) \quad \text{where} \quad a_k \neq 0.$$
Then $(-1)^k a_k \leq 0$.
\end{thm}
\addtocounter{thm}{-1}
}

We conjecture that an analogous rule holds for any link in $\PZ$ (where the sign is some fixed function of the leading degree $k$ and the number of components $m$).

There are other obstructions to membership in $\PZ$ that can be applied that we do not here investigate. These arise from branched covering arguments, $d$-invariants associated to branched covers, and $s$ or $\tau$ invariants of knots obtained as fusions of the components. The latter are discussed briefly in Section~\ref{sec:s}.

Contemporaneous work by Cha and Powell  investigates the entire family of relations $\geq_n$ for links ~\cite{ChaPow}. Their focus is on links that are slice in the topological category (and thus would resist many of the obstructions discussed here). Their results are striking, indicating that the filtration $\geq_n$ is highly non-trivial even when restricted to topologically slice links.

\section{Link inequalities and an operation preserving positivity}\label{sec:pos}

We first recall some terminology regarding links in $S^3$.
\begin{df}\label{def:concV}
Let $L = \left( L_{1} , \ldots, L_{n}\right), L' =  \left( L'_{1}, \ldots,  L'_{n}\right)$ be oriented $n$-component links in $S^3$ and let $V$ be a smooth, oriented, compact 4-manifold with $\partial V = S^3 \coprod -S^3$.  We say that \textbf{$L$ is concordant to $L'$ in $V$} if there exist oriented annuli $A_1, \ldots, A_n$ smoothly, disjointly, and properly embedded in $V$ and trivial in $H_2 \left( V,\partial V \right)$ such that for each $i$, $\partial A_i = L_i \coprod -L'_i \subset S^3 \coprod (-S^3)$.
\end{df}

\begin{df} \label{def:sliceV}
Let $L =  \left( L_{1} , \ldots, L_{n}\right)$ be an $n$-component link in $S^3$ and let $V$ be a smooth, oriented, 4-manifold with $\partial V = S^3$.  We say that \textbf{$L$ is slice in $V$} if there exist disks $D_1, \ldots, D_n$ smoothly, disjointly, and properly embedded in $V$ and trivial in $H_2(\left( V, \partial V \right)$ such that for each $i$, $\partial D_i = L_i $.
\end{df}

\begin{rmk}
Notice that the homological triviality condition on the concordance annuli (resp. slice disks) in Definition \ref{def:concV} (resp. \ref{def:sliceV}) implies that  $H_2(V)$ has a basis representable by surfaces disjointly and smoothly embedded in the complement of the annuli (resp. slice disks).
\end{rmk}

Cochran and Gompf  defined a relation $K \geq K'$ on knots which generalized the relation that $K$ can be transformed to $K'$ by changing only positive crossings ~\cite[Def. 2.1]{CGompf}. More recently Cochran-Harvey-Horn  generalized and filtered this relation to ``$\geq_n$'' ~\cite{CHH}.     We extend the $n = 0$ version to links in a straightforward way via the following definitions.

\begin{df}\label{def:geq}
Let $L,L' \subset S^3$ be $m$-component links.  We say that $L \geq_0 L'$ if $L$ is concordant to $L'$ in a smooth, simply-connected 4-manifold $V$ such that the intersection form on $H_2(V)$ is positive-definite.
%\begin{enum}
%\item $\pi_{1}(V) = 0$;
%\item the intersection form on $H_2(V)$ is positive-definite.
%%\item $H_{2}(V)$ has a basis represented by a collection of surfaces disjointly and smoothly embedded away from the concordance annuli.
%\end{enum}
\end{df}

\begin{rmk} Since $V$ is smooth and has $S^3$ boundary components, the intersection form on $H_2(V)$ is that of a closed smooth $4$-manifold. Thus it is unimodular and diagonalizable by Donaldson's theorem. It follows that $V$ is homeomorphic to a doubly-punctured connected sum of $\cp$'s, but possibly has an exotic smooth structure. The above relation clearly descends to one on $\mathcal{C}$, the smooth link concordance group.  While the relation is indeed reflexive and transitive, it fails to by symmetric. If $L\geq_0 L'$ then any sublink of $L$ is $\geq_0$ the corresponding sublink of $L'$. If $L\geq_0 L'$ then $-L'\geq_0 -L$ (here $-L$ denotes the reverse of the mirror image of $L$).
\end{rmk}

The following extends the notion of $0$-positive knots described in \cite{CHH}.

\begin{df}\label{def:ZP}
Let $L \subset S^3$ be an $m$-component link.  We say that \textbf{$L$ is $0$-positive}, or $L\in \PZ$, if $L$ is slice in a smooth, simply-connected 4-manifold $V$ such that the intersection form on $H_2(V)$ is positive-definite.
%\begin{enum}
%\item $\pi_{1}(V) = 0$;
%\item $V$ is positive-definite.
%%\item $H_{2}(V)$ has a basis represented by a collection of surfaces disjointly and smoothly embedded away from the slice disks.
%\end{enum}
The 4-manifold $V$  will be referred to as a \textbf{$0$-positon} for $L$.
\end{df}

Notice that $L$ is $0$-positive if and only if $L \geq_0 U$, where $U$ denotes the unlink.  As above, any $0$-positon is homeomorphic to a punctured connected sum of $\cp$'s, but possibly has an exotic smooth structure. We'll often abuse notation by writing $L \in \PZ$, meaning $\PZ$.  Denote by $\tld{\PZ}$ the set of concordance classes of links which are slice in a (non-exotic) $\text{closure}(\#_j \cp \setminus B^4)$ for any integer $j \geq 0$.  

It was observed in \cite{CHH} that if a knot $K_+$ can be obtained from another knot $K_-$ by changing a negative crossing to a positive one, then $K_+$ is concordant to $K_-$ in a doubly-punctured $\cp$.  It follows that is a knot $K$ admits a diagram with only positive crossings, then $K \in \tld{\PZ}$.  Consider the operation on knots consisting of inserting a full negative twist in an oriented band on some Seifert surface or, equivalently, changing a negative crossing into a positive as illustrated in Figure \ref{fig:pc}.  Performing $(+1)$-framed Dehn surgery on the dotted unknot in (A) effectively replaces the negative crossing in (A) with the positive one in (C).  We shall define a generalization of this operation on colored links.

\begin{figure}[h!]
\centering
\subfloat[]{
\includegraphics[height = 40mm]{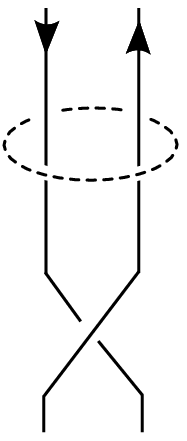}\label{fig:pc1}}
\put(10,50){$\rightsquigarrow$}
\put(10,54){$\rightsquigarrow$}
\subfloat[]{

\includegraphics[height = 40mm]{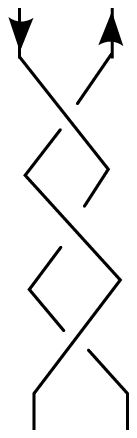}\label{fig:pc2}}
\put(-10,50){$=$}
\subfloat[]{

\includegraphics[height = 40mm]{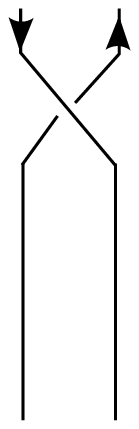}\label{fig:pc3}}
\caption{}\label{fig:pc}
\end{figure}

An \textbf{$m$-component $k$-colored link} $L$ is a link of $m$ components each of which has been assigned one of $k$ colors. This assignment is in general not assumed to be surjective. However if $L\in \mathcal{P}_0$ then we will implicitly assume the natural coloring where the $i^{th}$ component is assigned the color $i$.

\begin{df}\label{def:genposcross} The operation of \textbf{adding a generalized positive crossing} to a colored link is the transformation (A) $\rightsquigarrow$ (C) depicted in Figure \ref{fig:GPC}, where we require that, for each color, the union of the components of that color links algebraically zero times with the bold unknot in (B).  Obviously this is one of the Kirby moves together with an extra linking number restriction.
\end{df}

\begin{figure}[h!]
\centering
\subfloat[]{
\includegraphics[height = 40mm]{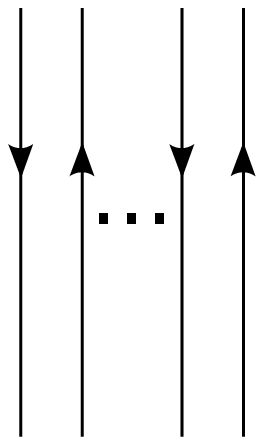}\label{fig:GPC1}}
\put(-5,50){$\rightsquigarrow$}
\put(-5,54){$\rightsquigarrow$}
\subfloat[]{

\includegraphics[height = 40mm]{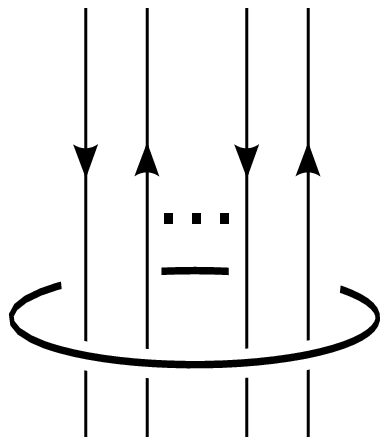}\label{fig:GPC2}}
\put(-10,50){$\cong$}
\put(-20,30){$+1$}
\subfloat[]{

\includegraphics[height = 40mm]{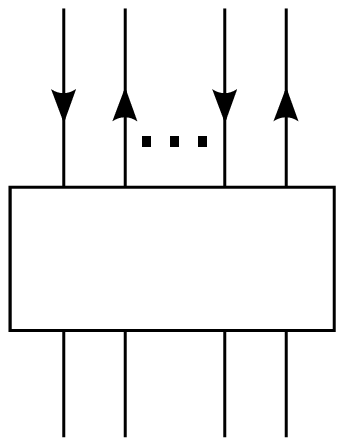}\label{fig:GPC3}}
\put(-60,48){$-1$}
\caption{}\label{fig:GPC}
\end{figure}

\begin{lem} \label{lem:closed underGPC} If $L'$ is a link which is obtained by adding a generalized positive crossing to a link $L$, then $L' \geq_0 L$.
\end{lem}
\begin{proof} Consider some diagram of $L$ which looks like Figure~\ref{fig:GPC1} in some local region, and let $L'$ be the result of adding a GPC in that region as shown in Figure~\ref{fig:GPC3}. There is a positive definite cobordism $C$ obtained from $S^3\times [0,1]$ by adding the single $(+1)$-framed two handle to the unknotted circle in $S^3\times\{1\}$ as shown in Figure~\ref{fig:GPC2}. Moreover $L$ is concordant to $L'$ in $C$. This shows that $L'\geq_0 L$. Note that the requirement that each component $L_i$ passes algebraically zero through the generalized crossing is necessary to establish that the annuli between $L_i$ and $L_i'$ are null-homologous.
\end{proof}
\begin{cor}\label{cor:GPC}
$\PZ$ and $\tld{\PZ}$ are closed under the operation of adding a generalized positive crossing.
\end{cor}
\begin{proof}
Lemma \ref{lem:closed underGPC} implies closure of $\PZ$ by transitivity of $\geq_0$.  However, notice that the cobordism $C$ appearing in the proof is \emph{diffeomorphic} to a twice punctured $\cp$.  So, if $L$ is slice in $\#_j \cp$, then in fact $L'$ is slice in $\#_{j+1} \cp$.
\end{proof}
\section{Fusions of boundary links}\label{sec:fusion}

\begin{df}\label{def:fusion} Suppose $L$ is an oriented link of $m$ components. A \textbf{fusion band} $B$ for $L$ is an embedding of an oriented $I\times I$ in $S^3$ for which $B\cap L= \partial I\times I$, $B$ connects two different components of $L$, and the orientation on $B$ is such that the $(m-1)$-component link given by $L- (B\cap L)\cup(I\times \partial I)$ can be given an orientation compatible with that of $L$. Such an $(m-1)$-component link is called a \textbf{fusion of $L$ corresponding to the band} $B$. More generally, given pairwise disjoint fusion bands $\{B_1,...,B_n\}$ the result of these simultaneous fusions is called a \textbf{fusion of $L$} if the resulting number of components is $m-n$ ~\cite[13.1]{Ka3}.
\end{df}

Note that the last hypothesis is equivalent to saying that, for each $1\leq i\leq n$, $B_i$ is a fusion band for the result of fusion along $\{B_1,...,B_{i-1}\}$. It is also equivalent to requiring that the graph whose vertices are the components of $L$ and whose edges are the core arcs, $( I\times \{1/2\})_i$ of the bands, is a ``forest'' (a disjoint union of one or more trees).
%This graph will be referred to as the \textbf{fusion graph}.

\begin{df}
Let $L \subset S^3$ be a $n$-component link.  Then we call $L$ \textbf{totally split} if one can find $n$ pairwise-disjoint open three-balls each containing one component of $L$.
\end{df}

\begin{prop}\label{prop:fusion} Let $L$ be a fusion of a boundary link.  Then there are totally split links $L_{1}$ and $L_{2}$ such that $L_{1} \geq_0 L \geq_0 L_{2}$.
\end{prop}
It is known that any integer-homology-$S^3$ can be obtained as surgery on a boundary link with all framings coming from the set $\{ \pm 1 \}$ ~\cite[Proposition 3.17]{CGO}\cite[Theorem A]{Mat87}.  Together with Proposition \ref{prop:fusion}, this immediately implies the following:
\begin{cor}\label{cor:fusion}
If $M$ is an integer homology three-sphere, then there are knots $K_1, K_2, \ldots, K_m \subset S^3$ and a four-manifold $W$ such that
\begin{enumerate}
\item $\partial W = M \sqcup \left( \#_{i=1}^m S_{n_i}^3(K_i)\right)$, where $n_i \in \{ \pm 1 \}$
\item $W$ has positive-definite diagonalizable intersection form
\item $H_1(W) = 0$
\end{enumerate}
Here $S_n^3(K)$ denotes $n$-framed surgery on the knot $K \subset S^3$.
\end{cor}

\begin{proof}[Proof of Proposition \ref{prop:fusion}] Suppose $L$ is a fusion of  $B$, a boundary link with components $B_1, \ldots, B_n$.  Let $\{F_i\}$ be a set of pairwise disjoint Seifert surfaces for the $B_i$.  Each $F_i$ may be viewed as a ``disk $D_i$ with bands''.  We may assume that the fusion bands attach only along $\partial D_i$ and otherwise intersect $F_i$ only in the interior of $D_i$.  The bands of the various $F_i$ can link with one-another arbitrarily.  However, by adding a sequence of GPC's in one of the two ways depicted in Figure \ref{fig:pass}, we can unlink the bands of $F_i$ from the bands of $F_j$ for $i\neq j$, thus transforming $B$ to some totally split link $B'$.  If we perform the very same moves on $L$, we arrive at a link $L'$, which is a fusion of the totally split link  $B'$.  Lemma \ref{lem:closed underGPC} implies that $L' \geq_0 L$.

\begin{figure}[h!]
\centering
\subfloat{
\labellist
\small
\pinlabel* {$+1$} at 75 30
\pinlabel* {$\cong$} at 170 70
\endlabellist
\includegraphics[height = 30mm]{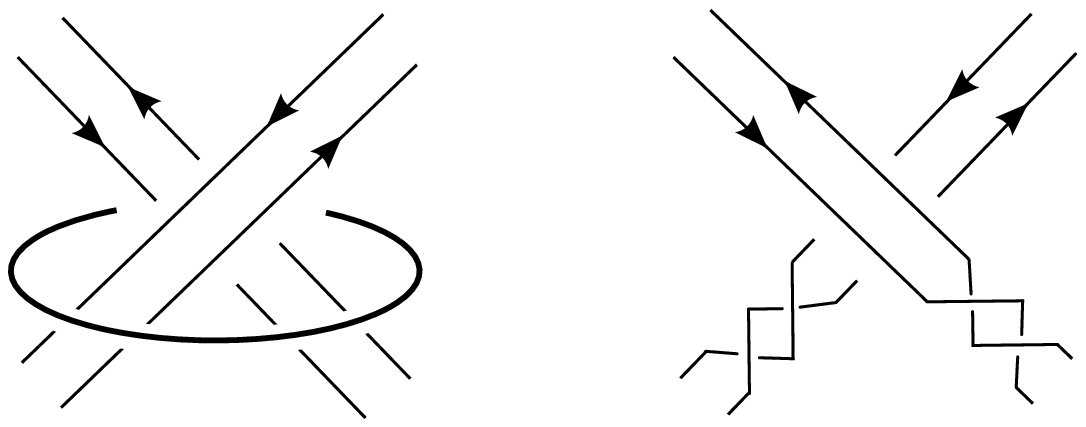}}
\subfloat{
\labellist
\small
\pinlabel* {$+1$} at 75 30
\pinlabel* {$\cong$} at 170 70
\endlabellist
\includegraphics[height = 30mm]{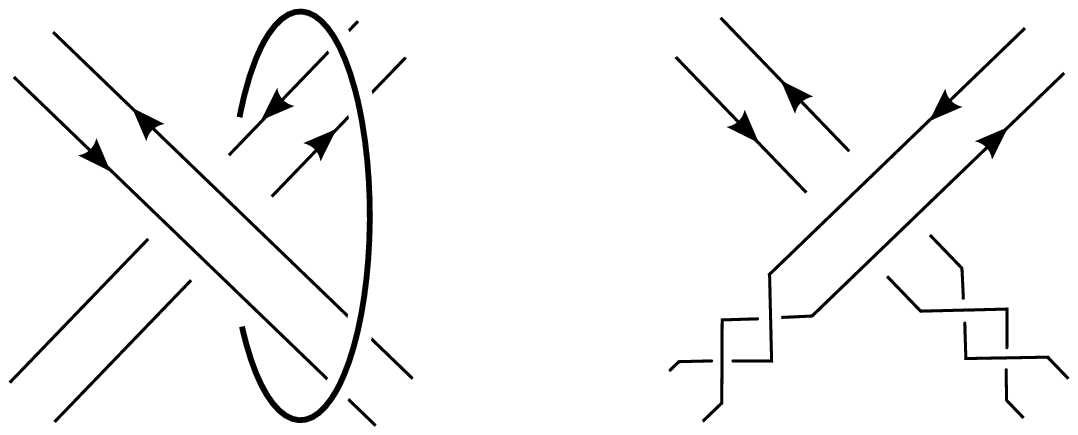}}
\caption{Two ways to change a crossing of oriented bands by adding a GPC}\label{fig:pass}
\end{figure}

In ~\cite{Miy}, Miyazaki showed that if a fusion of a totally split link produces a knot, then the smooth concordance class of that knot is independent of the band data - that is, his proof shows that \textit{a fusion of a totally split link is ribbon concordant to a connected sum}.  However, his proof doesn't rely on the number of components of the resulting fusion. Applying this we see that $L'$ is concordant to a link $L_1$ obtained by performing connected sums of several components of the totally split link $B'$, and so $L_1$ is itself totally split. Thus $L_1\geq_0 L$ where $L_1$ is totally split.

Applying the above process to $-L$ produces a totally split link $L_2$ with $-L_2 \geq_0 -L$.
\end{proof}

\section{Some obstructions arising from Milnor's invariants}\label{sec:milnor}

Propositon~\ref{prop:fusion} is suggestive. For, recall that the class of links that are concordant to fusions of boundary links is conjectured to be the same as the class of links all of whose Milnor's $\overline{\mu}$-invariants vanish ~\cite[Cor. 2.5]{C2}\cite[Question 16, p.66]{C4}\cite[p.572]{Le1}. This leads one to suspect that  the complete obstruction to a link's being $\geq_0$ a totally split link might be phrased in terms of Milnor's invariants. We proceed to investigate this possibility. Milnor's invariants may be loosely described as ``higher-order linking numbers'' (see ~\cite{C4} for a proof of this general meta-statement). Indeed the simplest non-zero Milnor's invariant of $L$, $\mubar_L(ij)$, is identifiable with the linking number between the $i^{th}$ and $j^{th}$-components, $\ell_{ij}$.

\begin{lem}\label{lem:linkingnumber}  If $L\geq_0 L'$ then  the pairwise linking numbers for $L$ and $L'$ are equal. In particular if $L$ is $\geq_0$ a totally split link then $\ell_{ij}=0$.
\end{lem}
\begin{proof} Let $A_i$ and $A_j$ be the annuli in $V$ as given by Definition~\ref{def:geq}. Let $F_j$, $-F_j'$ be Seifert surfaces for $L_j$ and $-L_j'$. Then of course $\ell_{ij}$ and  $\ell_{ij}'$, respectively,  are the algebraic intersection numbers $F_j \cdot L_i$ and $F_j'\cdot L_i'$. Let $B_j$ be the closed surface $A_j\cup -F_j\cup F_j'$. Since $A_i$ and $A_j$ are disjoint,  the algebraic intersection number $[B_j]\cdot [A_i]$  is equal to the difference $\ell_{ij}-\ell_{ij}'$. But this algebraic intersection number is zero since $[A_i]=0$ by hypothesis. Notice that we never used the hypothesis that $A_j$ and $A_i$ are annuli, and we needed only one of the two to be null-homologous.  This observation will be used in the proof of the next lemma.
\end{proof}

\textbf{Henceforth we restrict attention to  links with pairwise linking numbers zero.} Among such links the next simplest Milnor's invariants are the length three invariants, $\mubar_L(ijk)$ (for $i, j, k$ pairwise distinct). The integer $\mubar_L(ijk)$ depends only on the sublink $\{L_i,L_j,L_k\}$ and is equal to the negative of the linking number between $L_k$  and an oriented circle obtained as the intersection of Seifert surfaces for $L_i$ and $L_j$ ~\cite[p.71]{C4}. 

The following result is a consequence of a much more general result of C. Otto ~\cite[Thm.3.1, n=0]{Otto1}. We sketch a simple proof along the lines of the proof of Lemma~\ref{lem:linkingnumber}.

\begin{lem}[C. Otto]\label{lem:triplelinking}  If $L\geq_0 L'$ then for each $i\neq j\neq k$,  $\mubar_L(ijk)=\mubar_{L'}(ijk)$. Thus if  $L\geq_0$ a totally split link then, for each $i\neq j\neq k$,  the Milnor invariant  $\mubar_L(ijk)$ is zero.  
\end{lem}
\begin{proof}  It suffices to consider $3$-component links and show that $\mubar_L(123)=0$. For $i=1,2,3$ let $A_i$ be the disjoint annuli in $V$. Let $F_i$ and $F_i'$ be Seifert surfaces for $L_i$ and $-L_i'$. Since the pairwise linking numbers are zero these may be chosen in the link exteriors. We may assume that $F_1\cap F_2$ is a circle that we denote $L_{12}$  and that $F_1'\cap F_2'$ is a circle $L_{12}'$. Note that $\mubar_L(123)=-\ell k(L_{12},L_3)$ and $\mubar_{L'}(123)=- \ell k(L'_{12},L'_3)$.  For $i=1,2$ let $B_i$ denote the closed surface obtained by gluing $F_i$ and $F_i'$ onto $A_i$.   We claim that the $B_i$ bound $3$-manifolds $M_i$ in $V-A_3$. We only sketch the proof. The Seifert surfaces correspond, by the Pontryagin construction, to continuous maps on the link exteriors, $f:E(L)\to S^1\times S^1$ and $f':E(L')\to S^1\times S^1$, wherein the Seifert surfaces are obtained as the inverse image of the point $\{1\}\in S^1$ under the two projections to the circle. Since the annuli are null-homologous, the exterior of the union of all annuli in $V$ is a product on first homology. It follows that the maps $f$ and $f'$ extend to a map on the exterior of the annuli. The $3$-manifolds are then obtained as the inverse images of $\{1\}$ under the two projection maps to the circle. Then $M_1\cap M_2$ is an oriented surface $A_{12}$ disjoint from $A_3$ in $V$. It follows from the proof of Lemma~\ref{lem:linkingnumber}, that the linking numbers of the boundaries of these surfaces are equal. Here we use the fact that $[A_3]=0$, but we do not need $A_{12}$ to be an annulus, nor to be null-homologous, as observed in the proof of Lemma~\ref{lem:linkingnumber}.
\end{proof}

The next most complicated Milnor's invariants are of length four. In particular, for links of $2$ components with linking number zero there is only one, $\mubar(1122)$, which equals $-\beta(L)$ where $\beta(L)$ is the \textbf{Sato-Levine invariant} of the 2-component link $L$ ~\cite[p.71]{C4}\cite[Thm.9.1]{C3}\cite[Section 4]{C5}. The latter is, by definition, the self-linking of the circle $L_{12}$ obtained as the intersection, $F_1\cap F_2$, of two Seifert surfaces, with respect to a push-off using a normal vector field to either surface.

\begin{prop}\label{prop:SL1}
Let $L,L' \subset S^3$ be $2$-component links.  If $L \geq_0 L'$ then $\mubar_L(1122)\geq \mubar_{L'}(1122)$ (alternatively $\beta(L) \leq \beta(L')$). 
\end{prop}
\begin{cor}\label{cor:SL1}
Let $L$ is $\geq_0$ a totally split link then for all $i,j$, $\mubar_L(iijj)\geq 0$  (alternatively the Sato-Levine invariant of each $2$-component sublink is non-positive).
\end{cor}

\begin{proof}  For $i = 1, 2$, let $F_i$, $F_i'$, $A_i$, $A_i'$ be as in the proof of Lemma \ref{lem:triplelinking} and let $L_{12} = F_1 \cap F_2$ and $L'_{12} = F'_1 \cap F'_2$.  As above let $B_i$ be constructed by gluing $F_i$ and $F'_i$ onto $A_i$. Again construct 3-manifolds $M_i$ bounding the $B_i$ and let $A_{12} = M_1 \cap M_2$.  Let $F_{12}$ and $F_{12}'$ be Seifert surfaces for $L_{12}$ and $L_{12}'$. Let $B_{12}$ be constructed by gluing $-F_{12}$ and $F'_{12}$ onto $A_{12}$. Let $A^+_{12}$ denote the pushoff of $A_{12}$. On the one hand we have that the self-intersection number of $B_{12}$ is non-negative because the intersection form of $V$ is positive-definite. On the other hand, this intersection number may be calculated as the algebraic count of the transverse intersections of $B_{12}$ with $A^+_{12}$. Since $A^+_{12}$ is disjoint from $A_{12}$, this intersection number is equal to 
$$
-F_{12}\cdot A_{12}^++F'_{12}\cdot A_{12}^+=-\ell k(L_{12}, L^+_{12}) +\ell k (L'_{12}, (L'_{12})^+) = -\beta(L) + \beta(L').
$$
Thus $\beta(L)\leq \beta(L')$.
\end{proof}

For very simple $2$-component links, the conditions above completely determine whether or not that link is $\geq_0$ a split link, as the next result shows.

\begin{prop}\label{prop:SL2}
Suppose $L$ is a 2-component link for which each component admits a genus 1 Seifert surface in the exterior of the other. Then $L$ is  $\geq_0$ a split link if and only if $\mubar_L(1122)\geq 0$ (alternatively  $\beta(L) \leq 0$).
\end{prop}

\begin{figure}[h!]
\centering
\subfloat[The Seifert surfaces bounded by the components $L_i$ - the dashed green curves are identified, as they both represent $\gamma$.]{
\label{fig:2comp1}
\labellist
\small
\pinlabel* {$K$} at 70 135
\pinlabel* {$K$} at 240 135
\pinlabel* {\rotatebox{50}{$-n$}} at 35 122
\pinlabel* {\rotatebox{50}{$-n$}} at 205 122
\endlabellist
\includegraphics[height = 36mm]{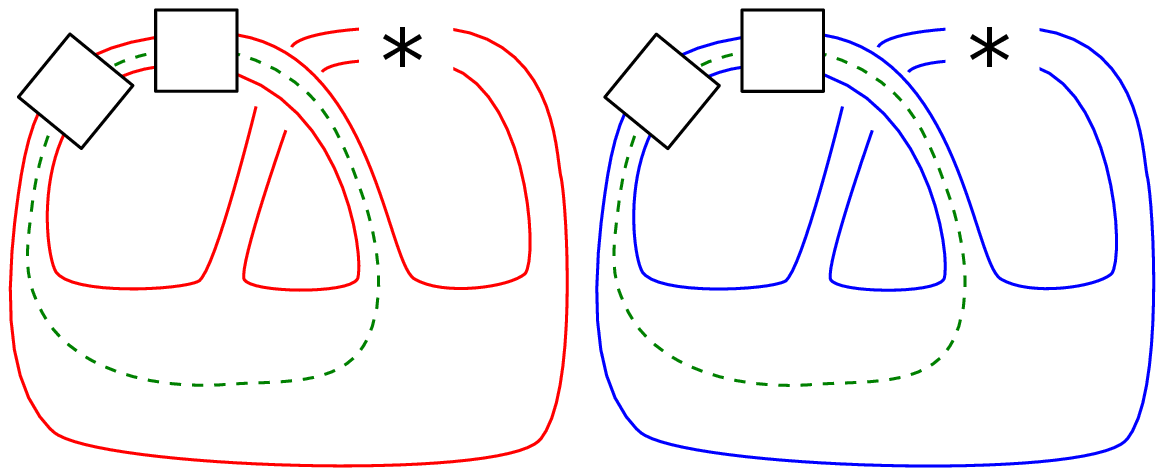}}
\subfloat[$L$ can be obtained from the link above by fusing the components with matching colors.]{
\label{fig:2comp2}
\labellist
\small
\pinlabel* {$K$} at 280 170
\pinlabel* {$-n$} at 275 115
\endlabellist
\includegraphics[height = 36mm]{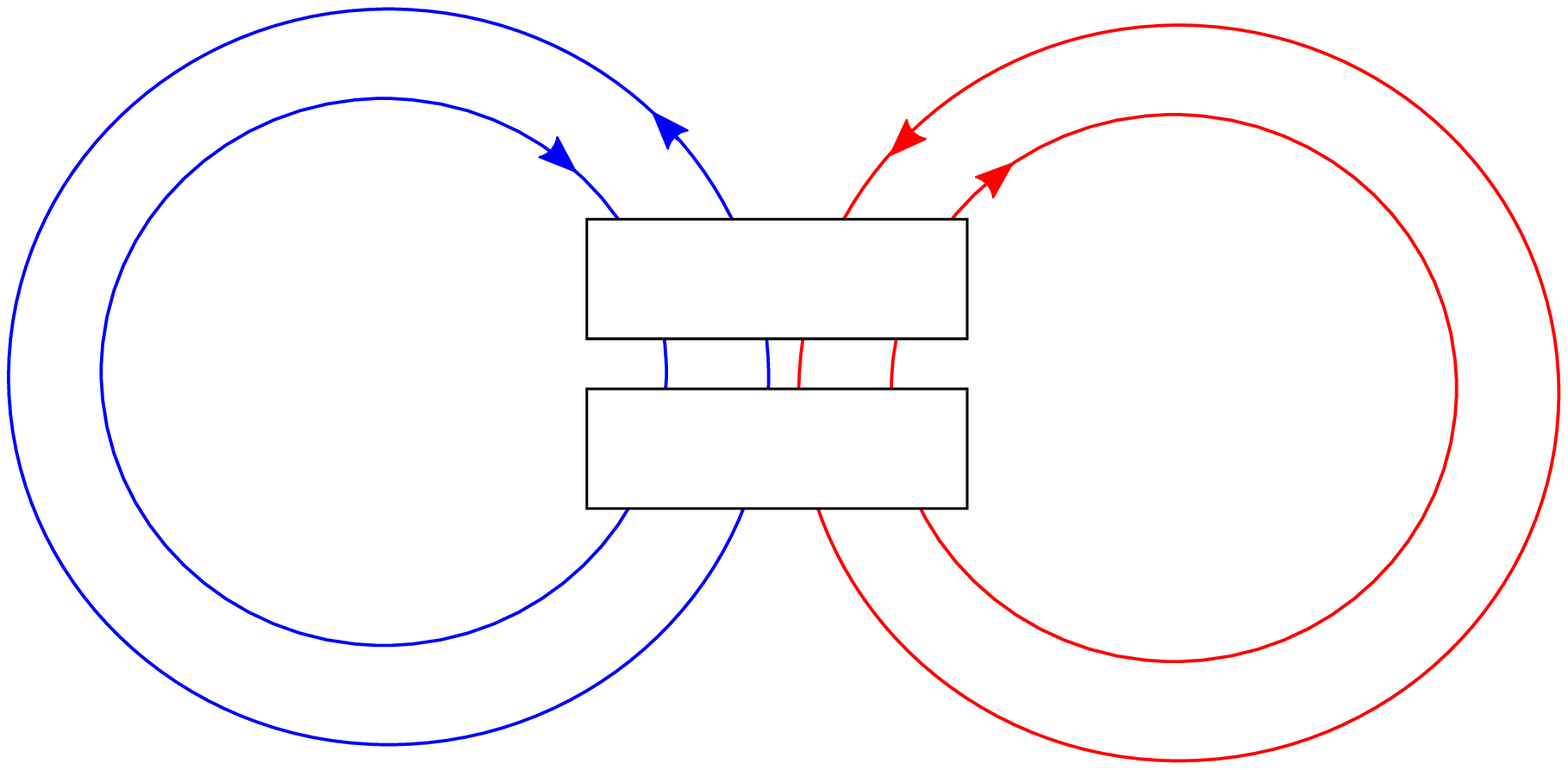}}
\caption{Two views of the link $L$ in the proof of Proposition \ref{prop:SL2}.}
\end{figure}

\begin{proof} One implication follows from Corollary~\ref{cor:SL1}. For the other implication suppose $\beta(L) \leq 0$. By Propositon~\ref{prop:fusion} it suffices to show that $L\geq_0$ a fusion of a boundary link.  There are Seifert surfaces for $L_i$ as shown in Figure \ref{fig:2comp1}, which intersect in the simple closed curve $\gamma$.  A ``$-n$`` box indicates $n$ left-handed twists in the strands, a ``$K$'' box indicates that the strands are tied in the knot $K$ in a zero-twisted fashion, and the symbol ``$*$'' indicates that a band twists, knots, and links with other bands arbitrarily.  One could instead then view $L$ as a fusion of the four-component link shown in Figure \ref{fig:2comp2}, where one band fuses together the red components and another the blue components.  Let $L'$ denote the link obtained from $L$ by omitting the box of twists - $L'$ is in fact a fusion of the boundary link consisting of four parallel copies of $K$.  It's easy to see that $L$ can be obtained from $L'$ by adding $n$ GPC's, and so $L \geq_0 L'$
\end{proof}

In ~\cite{C4} the first author greatly generalized the Sato-Levine invariant of a $2$-component link $L$ by defining  two sequences of integral invariants
$\{\beta^n(L)\}$ and $\{\overline{\beta}^n(L)\}$, $n\geq 1$ that are independent except for $n=1$ where they each agree with the Sato-Levine invariant. That is,  $\beta^1(L)=\overline{\beta}^1(L)=\beta(L)$. These sequences are defined as follows. Given $L=(L_1,L_2)$ one defines a new $2$-component link $D(L)$, \textbf{the derivative of $L$ with respect to the first component}, to be the link $(L_{12},L_2)$ where $L_{12}$ is the oriented circle obtained as the intersection, $F_1\cap F_2$, of some Seifert surfaces for $L_1$ and $L_2$. Then $\beta^1(L)$ is defined to be the Sato-Levine invariant of $L$, and for $n>1$, the numbers $\beta^n(L)$ are defined recursively via $\beta^n(L)\equiv\beta^{n-1}(D(L))$.  In other words, the invariant $\beta^n(L)$ is the Sato-Levine invariant of the derivative  $D^{n-1}(L) = D(D(\ldots D(L)) \ldots )$, the result of applying the derivative operation $n-1$ times. The isotopy class of $D(L)$ is not well-defined  of course since there are many choices for the $F_i$. Nonetheless the numbers $\beta^n(L)$ are well-defined. Note that these derived links are obtained by always retaining the second component. By retaining the first component, one defines a derivative with respect to the second component and hence a second sequence $\overline{\beta}^n(L)$. For simplicity we suppress the latter since by symmetry all of our results hold for this other sequence. These sequences were later shown to be integral lifts of Milnor's invariants $\{(-1)^n\mubar_L(1122....22)\}$ and $\{(-1)^n\mubar_L(11...1122)\}$ of length $2n+2$ ~\cite[Thm.6.10]{C4}.

We can now greatly generalize Proposition~\ref{prop:SL1}:

\begin{thm}\label{thm:betan}
Let $L,L' $ be $2$-component links with zero linking number such that $L \geq_0 L'$. If $\beta^i(L)=\beta^i(L')$ for $1\leq i < n$, then $\beta^n(L) \leq \beta^n(L')$. 
\end{thm}

\begin{cor}\label{cor:betan}
If $L$ is $\geq_0$ a totally split link, then for each two-component sublink $L_{ij}$, the first non-zero invariant $\beta^n(L_{ij})$ invariant is non-positive.
\end{cor}

\begin{figure}[h!]
\centering
\begin{minipage}[c]{.35\textwidth}
\labellist
\small
\pinlabel* {$m$} at 345 164
\pinlabel* {$n$} at 170 280
\endlabellist
\includegraphics[height = 40mm]{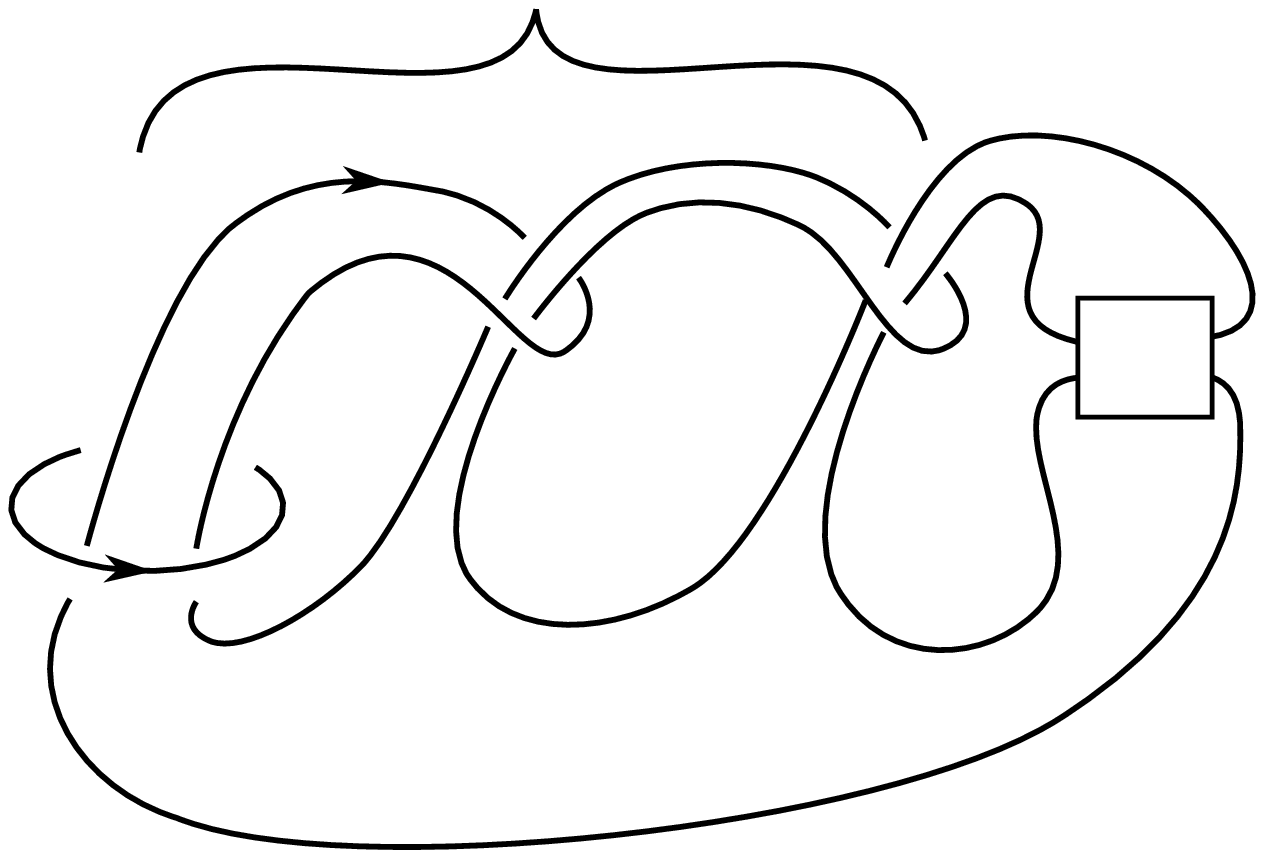}
\end{minipage}
\begin{minipage}[c]{.6\textwidth}
\caption{The two-component link $M(n,m)$.  The box represents $m$ full twists.}\label{fig:milnor}
\end{minipage}
\end{figure}

\begin{ex}\label{ex:milnor}
Consider the two component link $M(n,m)$ depicted in Figure \ref{fig:milnor} (where $m,n \in \mathbb{Z}$ with $n \geq 0$ and $m \neq 0$).   It's straightforward to verify that $\beta^{k}(M(n,m)) = 0$ for $1 \leq k \leq n$ while $\beta^{(n+1)}(M(n,m)) = m$.  Thus if $m > 0$ then $\beta^{(n+1)}(M(n,m))>0$ so, by Corollary~\ref{cor:betan}, $M(n,m)$ is not $\geq_0$ a totally split link and hence not $0$-positive. But in this case the crossings inside the box in Figure~\ref{fig:milnor} are negative crossings. Since $M(n,m)$ can be changed to a trivial link by changing some of these crossings, it is $0$-negative. This example shows that the sign of  Milnor's invariants of arbitrarily large length can be involved in the question of whether or not a link is $\geq_0$ a split link.

Moreover these examples also show that the hypothesis of Proposition \ref{prop:SL2} regarding genus one Seifert surfaces is necessary.  It cannot even be relaxed to merely assume that each component is knot of genus at most 1.  For the components of the link $M(1,1)$ are an unknot and a figure-eight knot, respectively, and $\beta(M(1,1)) = \beta^1(M(1,1)) = 0$.  However, $\beta^2(M(1,1)) =1$ and thus $M(1,1)$ is not $\geq_0$ a split link.
\end{ex}

Theorem \ref{thm:betan} will be proved without much work by relaxing the relation $\geq_0$ to  a relation $\geq_0^w$ which is respected by the derivative operation. This will enable  an inductive proof.

\begin{df}\label{def:weakcobV} The links $L = (L_{1},L_{2})$ and $L' = ( L'_{1},L'_{2} )$ are \textbf{weakly cobordant in $V$}  if there  a compact oriented \textit{surface}, $ A_1$ and an oriented annulus $A_2$,  smoothly, disjointly, and properly embedded in $V$ such that for each $i$, $\partial A_i = L_i \coprod -L'_i$, $A_i$ is trivial in $H_2\left(V, \partial V \right)$, and moreover there is a null homology for $A_2$ in the exterior of $A_1$.
\end{df}

\begin{df}\label{def:geqweak} We say that $L \geq_0^w L'$ if $L$ is weakly cobordant to $L'$ in a smooth, simply-connected 4-manifold $V$ such that the intersection form on $H_2(V)$ is positive-definite.
%\begin{enum}
%\item $\pi_{1}(V) = 0$;
%\item the intersection form on $H_2(V)$ is positive-definite;
%
%\end{enum}
\end{df}

Therefore, compared to $\geq_0$, we have greatly relaxed the requirement that $A_1$ be an annulus but imposed a stronger null-homology property for $A_2$. Proposition~\ref{cor:betan} will follow from the following two lemmas. First we generalize Proposition~\ref{prop:SL1}.

\begin{lem}\label{lem:weakSL1}
Let $L,L' \subset S^3$ be $2$-component links.  If $L \geq_0^w L'$ then  $\beta(L) \leq \beta(L')$ (or $\mubar_L(1122)\geq \mubar_{L'}(1122)$). 
\end{lem}

Then we prove:

\begin{lem}\label{lem:derivatives}
 If $L \geq_0^w L'$ and $\beta(L)=\beta(L')$ then  $D(L) \geq_0^w D(L')$.
\end{lem}

\begin{proof}[Proof of Theorem~\ref{thm:betan} assuming Lemmas~\ref{lem:weakSL1} and ~\ref{lem:derivatives}]
%Let $L,L' $ be $2$-component links such that $L \geq_0 L'$ and $\beta^i(L)=\beta^i(L')$ for $1\leq i < n$. Then in particular $L$ and $L'$ satisfy the hypotheses of Lemma~\ref{lem:derivatives}, so $D(L) \geq_0^w D(L')$. Hence, by Lemma~\ref{lem:weakSL1},
%\begin{equation}\label{eq:beta2}
%\beta^2(L)\equiv \beta(D(L)\leq \beta(D(L')\equiv \beta^2(L').
%\end{equation}
%Therefore if $n=2$, we are done. Suppose $n>2$. Then by hypothesis $\beta^2(L)=\beta^2(L')$ so, by definition, $\beta(D(L))=\beta(D(L'))$. Then $D(L)$ and $D(L')$ satisfy the hypotheses of Lemma~\ref{lem:derivatives} implying that $D^2(L) \geq_0^w D^2(L')$.  Hence, by Lemma~\ref{lem:weakSL1},
%\begin{equation}\label{eq:beta3}
%\beta^3(L)\equiv \beta(D^2(L)\leq \beta(D^2(L')\equiv \beta^3(L').
%\end{equation}
%Thus if $n=3$ we are done and so on.

We proceed by induction on the parameter $n$.  If $n=1$, then the result follows immediately from Lemma \ref{lem:weakSL1}.  For some $k \geq 1$, we assume that the statement holds with $n=k$ (for any pair of links $L$ and $L'$ satisfying the hypotheses).  We now prove the statement for $n = k+1$.  Let $L$ and $L'$ be links with zero linking number such that $L \geq_0^w L'$ and $\beta^i(L) = \beta^i(L')$ for $1 \leq i < k+1$. Then notice that
\begin{equation*}
\beta^i(D(L)) \equiv \beta^{i+1}(L) = \beta^{i+1}(L') \equiv \beta^i(D(L')) \quad \text{for} \quad 1 \leq i < k.
\end{equation*}
Additionally, $D(L) \geq_0^w D(L')$ by Lemma \ref{lem:derivatives}, and so the $n=k$ statement applied to the links $D(L)$ and $D(L')$ implies that 
\begin{equation*}
\beta^{k+1}(L) \equiv \beta^{k}(D(L)) \leq \beta^{k}(D(L')) \equiv \beta^{k+1}(L')
\end{equation*}
%Let $L \geq_0 L'$ and let assume that $\beta^i(L)=\beta^i(L')$ for $1\leq i < n$.  If $n=1$, the result follows from Lemma \ref{lem:weakSL1}.  If $n \geq 2$, we claim that $D^{n-1}(L) \geq_0^w D^{n-1}(L')$.  If this holds, then Lemma \ref{lem:weakSL1} implies that we're done, since $\beta^{n}(L) \equiv \beta(D^{n-1}(L))$ and $\beta^{n}(L') \equiv \beta(D^{n-1}(L'))$.  We prove the claim by induction on $n$.
%
%The $n=2$ case follows immediately from Lemma \ref{lem:derivatives}.  If $n > 2$, then we can use Lemma \ref{lem:derivatives} along with the induction hypothesis to finish the proof, since $D^{n-1}(L) \equiv D(D^{n-2}(L))$, $D^{n-1}(L') \equiv D(D^{n-2}(L'))$, and
%$$\beta(D^{n-2}(L)) \equiv \beta^{n-1}(L) = \beta^{n-1}(L') \equiv \beta(D^{n-2}(L'))$$
\end{proof}

\begin{proof}[Proof of Lemma~\ref{lem:weakSL1}]  In looking at the proof of Proposition~\ref{prop:SL1}, one sees that we never used the fact that the $A_i$ were annuli. All we needed was the existence of the $3$-manifolds $M_i$. This follows if  $B_1$ is null-homologous in the exterior of $A_2$ and $B_2$ is null-homologous in the exterior of $A_1$. Under the hypothesis that $L \geq_0^w L'$, the second of these two is explicitly given in Definition~\ref{def:weakcobV}. From  this definition we also know that $A_1$ is zero in $H_2(V,\partial V)$ and hence that $B_1$ is zero in $H_2(V)$. However,  we claim that the map $i_*:H_2(V-N(A_2))\to H_2(V)$ is injective (which would finish the proof). For, from the exact sequence
$$
H_3(V,V-N(A_2))\overset{\partial_*}{\longrightarrow}H_2(V-N(A_2))\overset{i_*}{\longrightarrow}H_2(V)
$$
we know that $B_1$ is in the image of $\partial_*$. But by excision (note that since $A_2$ is a null-homologous annulus  it has a trivial normal bundle)
$$
H_3(V,V-N(A_2))\cong H_3(A_2\times D^2,A_2\times\partial D^2)\cong H_3(S^1\times D^2,S^1\times S^1)\cong\mathbb{Z},
$$
generated by the solid torus neighborhood of $L_2$ in $S^3$. The image of this class under $\partial_*$ is the boundary of the regular neighborhood of $L_2$ which is null-homologous in $S^3-N(L_2)$ and hence zero in $H_2(V-N(A_2))$. Thus $B_1$ is null-homologous in the \textit{exterior} of $A_2$.
\end{proof}

\begin{proof}[Proof of Lemma~\ref{lem:derivatives}] For $i =1,2$ let $F_i$, $F_i'$ be Seifert surfaces with connected intersection $L_{12} = F_1 \cap F_2$ and $L'_{12} = F'_1 \cap F'_2$. For $i =1,2$ let $B_i$ be the closed surfaces defined by gluing $-F_i$ and $F'_i$ onto $A_i$. By hypothesis (Definition~\ref{def:weakcobV}) $B_2$ is null-homologous in the exterior of $A_1$ and  $A_1$ is zero in $H_2(V,\partial V)$. But, as in the proof of Lemma~\ref{lem:weakSL1}, since $A_2$ is an annulus, it follows that $B_1$ is null-homologous in the \textit{exterior} of $A_2$. Thus again there exist a 3-manifold $M_1$ bounding $B_1$ in the exterior of $A_2$ and there exists a 3-manifold $M_2$ bounding $B_2$ in the exterior of $A_1$. Let $A_{12} = M_1 \cap M_2$, an oriented surface with trivial normal bundle in the exterior of $A_2$. A push-off $A_{12}^+$  does not intersect $M_2$. The link $(L_{12}^+,L_2)$ is isotopic to $D(L)$ and the link $((L'_{12})^+,L_2)$ is isotopic to $D(L')$. Assuming these identifications, we claim that $(A^+_{12},A_2)$ is a weak cobordism in $V$ between these two links. Since $M_2$ is a null-homology for $B_2$ in the exterior of $A^+_{12}$, the only thing to show is that $A_{12}^+$ is zero in $H_2(V,\partial V)$. Let $F_{12}^+$ and $(F'_{12})^+$ be Seifert surfaces for $L_{12}^+$ and $(L'_{12})^+$. Since $H_2(V)\cong H_2(V,\partial V)$, it suffices to show that the closed surface $B_{12}^+=A_{12}^+\cup -F_{12}^+ \cup (F'_{12})^+$ is zero in $H_2(V)$. Since the intersection form is positive definite, this is equivalent to showing that $B_{12}\cdot B_{12}=0$. The latter may be computed by counting intersections between $B_{12}$ and $A_{12}^+$. Since $A_{12}\cap A_{12}^+=\emptyset$, this intersection is equal to $-F_{12}\cdot L_{12}^+$ plus $F'_{12}\cdot (L'_{12})^+$. The latter are the self-linking numbers of $L_{12}$ and $L'_{12}$ respectively. Thus 
$$
B_{12}\cdot B_{12}=-\beta(L)+\beta(L')=0,
$$ 
by hypothesis. Thus we have shown that $D(L)$ is weakly cobordant to $D(L')$ in $V$ so $D(L)\geq^w_0 D(L')$.
\end{proof}

At this point it might be reasonable to conjecture that if $L\in \PZ$ and Milnor's $\mubar$-invariants of length less than $n$ vanish, then for all sequences $I$ of length $n$, $\mubar_L(I)$ is non-negative (or non-positive according to the parity of $n$). But such a result is not likely to be true because, for example, $\PZ$ is closed under changing the orientation of the $i^{th}$- component, while the effect on $\mubar(I)$ is to change by $(-1)^{n_i}$ where $n_i$ is the number of occurences of $i$ in $I$ ~\cite{M2}. Example~\ref{ex:not nh} below confirms that this conjecture is false. A correct conjecture would have to be more complicated.

\begin{ex}\label{ex:not nh} Recall that two $n$-component links $L_0,L_1 \subset S^3$ are  \textbf{link homotopic}  (or sometimes merely \textbf{homotopic}) if there is a homotopy $H:\left( \coprod_{k=1}^n S^1\right) \times [0,1] \rightarrow S^3$ such that $H\left(\coprod_{k = 1}^n S^1\right) \times \{i\}) = L_i$ for $i = 0,1$ and such that the images of the circles are pairwise-disjoint for each $t \in [0,1]$ ~\cite{M1}.  In other words, the components may cross themselves during the homotopy but may not cross one another. Recall also that concordant links are link homotopic, and in particular any slice link is link homotopic to the unlink. $\PZ$ is closed under concordance but not closed under link homotopy since any two component link with zero linking number is link homotopic to the trivial link, but membership in $\PZ$ is obstructed by, for example, the signs of $\beta^n(L)$. Link homotopy is controlled by $\mubar_L(I)$ where the sequence $I$ contains no repeated indices. So for example the invariants $\mubar(ij)$ and $\mubar(ijk)$ that we have discussed are invariants of link homotopy ~\cite{M2}. Since these are known to be the only invariants of link homotopy for $2$ and $3$-component links, it follows from Lemmas~\ref{lem:linkingnumber} and ~\ref{lem:triplelinking},  that for such links being $0$-positive implies being link homotopic to a trivial link.  But this pattern does not continue. The following example shows that there are $4$-component links in $\PZ$ that have $\mubar(1234)\neq 0$. Hence being $0$-positive does not imply that $L$ is null-homotopic. Let $L$ denote the four-component link appearing in Figure \ref{fig:ex4}.  Clearly $L \in \tld{\PZ}$ since it is obtained from the trivial link by the addition of a generalized positive crossing.  However, $L$ is link homotopic to the Bing double of a Hopf link which is known to have $\mubar_L(1234)=\pm 1$ ~\cite[Figure 7, n=4]{M1}. Thus $L$ is not link homotopic to the trivial link. Moreover, by changing the orientation of the first component we can make $\mubar_L(1234)$ achieve both signs.
\begin{figure}[h!]
\centering
\begin{minipage}[c]{.4\textwidth}
\labellist
\small
\pinlabel* {$-1$} at 250 130
\endlabellist
\includegraphics[width = 70mm]{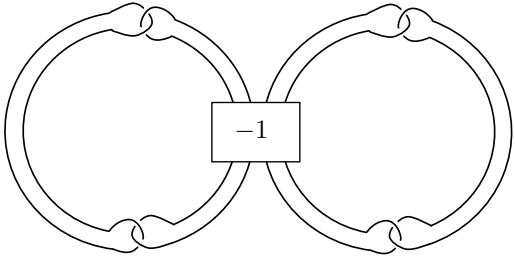}
\end{minipage}
\begin{minipage}[c]{.55\textwidth}
\caption{This zero-positive link is not null-homotopic\label{fig:ex4}}
\end{minipage}
\end{figure}

\end{ex}

\section{A characterization of links that are slice in a connected sum of $\cp$'s}\label{sec:fam}

In this section we give two characterizations, up to concordance, of the set $\tld{\PZ}$.  If $\#_j \cp$, for each $j$, has a unique smooth structure (which is unknown at this time),  this would completely characterize  $\mathcal{P}_0$.   We first define a distinguished family of $0$-positive links that generalize the positive Hopf link.
\begin{df} \label{df:nghl} A \textbf{$k$-colored null generalized Hopf link (or NGHL)} of $2n$-components is a colored link $L$ obtained by taking $2n$ parallel fibers of the Hopf fibration, orienting $n$ of them in each direction, and finally assigning each component with one of $k$ colors such that the total algebraic count of fibers representing each color is equal to zero. 
\end{df}

When $n=1$, one obtains the positive Hopf link where both components have the same color.  Figure \ref{fig:nghl} illustrates the general form of a two-colored NGHL.

\begin{rmk}\label{rmk:gpc nghl}
Notice that a $k$-colored NGHL of $2n$ components is simply the result of adding a GPC to a $2n$-component $k$-colored unlink in such a way that each component links the $(+1)$-framed unknot achieving the GPC addition geometrically once.
\end{rmk}

\begin{figure}[h!]
\centering
\begin{minipage}[c]{.35\textwidth}
\labellist
\small
\pinlabel* {$-1$} at 270 135
\pinlabel* {$n_1$ pairs} at 150 155
\pinlabel* {$n_2$ pairs} at 410 155
\endlabellist
\includegraphics[height = 32mm]{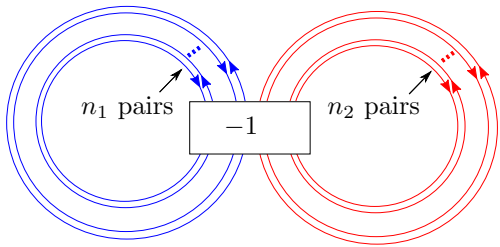}
\end{minipage}
\begin{minipage}[c]{.6\textwidth}
\caption{A two-colored null generalized Hopf link\label{fig:nghl}}
\end{minipage}
\end{figure}

%\begin{rmk}\label{rem:Hopffib} A $k$-colored $2n$-component null generalized Hopf link (in which all components are involved in the crossing change) is the same as the link obtained from taking $2n$ parallel copies of the fiber of the Hopf fibration, then orienting $n$ of them one way and $n$ of them the other way, and finally assigning colors making sure that each color corresponds to an algebraically zero number of fibers. Such links also arise naturally as the links of certain singularities. 

\begin{rmk}\label{rem:Hopffib}  NGHL's arise naturally as the links of certain singularities.  Namely consider a disjoint collection $\{D_1,...,D_k\}$ of oriented connected null-homologous surfaces embedded disjointly in $\cp$ in such away that each meets the exceptional curve $E=\mathbb{C}P_1$ transversely as shown schematically in Figure~\ref{fig:bd1}. The condition that $D_i$ is  null-homologous means that it intersects $E$ in $2n_i$ points where $n_i$ are positive intersections and $n_i$ are negative intersections. 
\begin{figure}[h!]
\centering
\subfloat[]{
\includegraphics[height = 40mm]{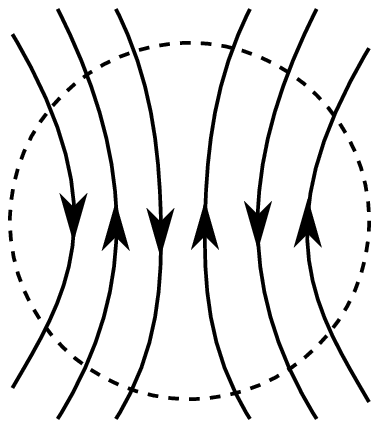}\label{fig:bd1}}
\subfloat[]{
\includegraphics[height = 40mm]{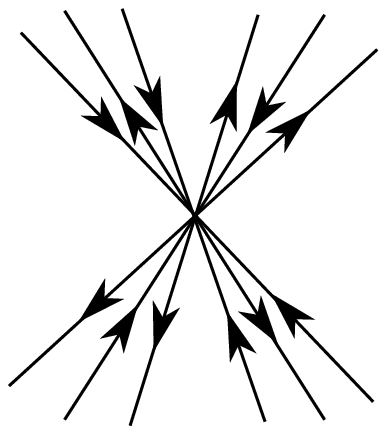}\label{fig:bd2}}
\caption{Blowing down a $\cp$ to obtain a singularity of the slice disk(s).  The exceptional sphere is dotted.}\label{fig:bd}
\end{figure}
If we were to blow down the $\cp$ we would get a $k$-colored singularity in $B^4$, as shown schematically in Figure~\ref{fig:bd2}, where the link of this singularity is a $k$-colored link, $\mathcal{L}$, in $S^3$ with $2n=2n_1+...2n_k$-components. We claim that $\mathcal{L}$ is precisely a $k$-colored $2n$-component NGHL as in the previous paragraph. This is seen as follows. Let $N(E)$ be a tubular neighborhood of $E$ which may be identified with the total space of the normal $2$-disk bundle to $E$. It is known that the circle bundle $\partial N(E)\to E$ is precisely the Hopf fibration $S^3\to S^2$. Then $\mathcal{L}$ is the same the intersection of the union of the $D_i$ with $\partial N(E)\cong S^3$. Thus the components of $\mathcal{L}$ are parallel fibers of the Hopf fibration, colored with the color $i$ if they arise from $D_i\cap E$. Since $[D_i]\cdot [E]$=0, $n_i$ of the $i$-colored components are oriented one way and $n_i$ are oriented the other way.
\end{rmk}

\begin{df} An $n$-component link $L$ (viewed as an $n$-colored link in the canonical manner) is a \textbf{fusion of NGHL's} if it is obtained as a fusion of a disjoint union of $n$-colored NGHL's where the fusion bands connect only components of the same color. 
\end{df}

The following observation will be useful.
\begin{lem}\label{lem:nghl+triv}
The disjoint union of a $k$-colored NGHL with a $k$-colored trivial link of $n$ components is a fusion of another $k$-colored NGHL.
\end{lem}
\begin{proof}
Let $L$ be the disjoint union of a $k$-colored NGHL and a $k$-colored trivial link of $n$ components, and consider a diagram for $L$ which appears as in Figure \ref{fig:unknot1} in the region of the GPC addition.  Figure \ref{fig:unknot3} is a diagram of another $L'$ we assume that the diagrams for $L$ and $L'$ agree outside of the regions depicted in the figures - notice in particular that $L'$ is a disjoint union of a $k$-colored NGHL and a $k$-colored trivial link of $(n-1)$ components.  Performing the fusion indicated by the dashed segment yields the diagram for $L$ depicted in Figure \ref{fig:unknot2}.  Since ``is a fusion of'' is a transitive relation on links, this proves the lemma.
\begin{figure}[h!]
\centering
\subfloat[$L$]{
\labellist
\small
\pinlabel* {$-1$} at 95 105
\endlabellist
\includegraphics[height = 28mm]{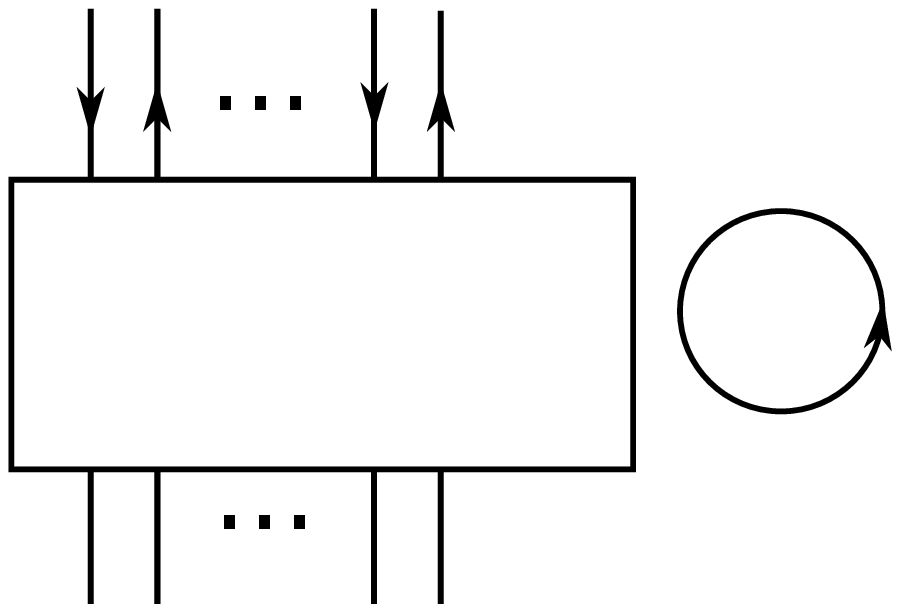}\label{fig:unknot1}}
\subfloat[$L'$]{
\labellist
\small
\pinlabel* {$-1$} at 95 105
\endlabellist
\includegraphics[height = 28mm]{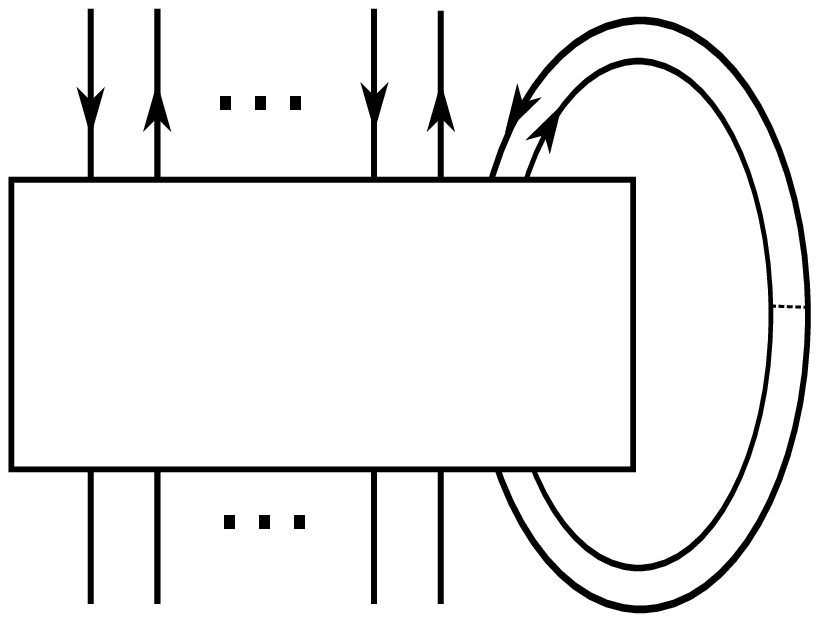}\label{fig:unknot3}}
\subfloat[$L$]{
\labellist
\small
\pinlabel* {$-1$} at 95 105
\endlabellist
\includegraphics[height = 28mm]{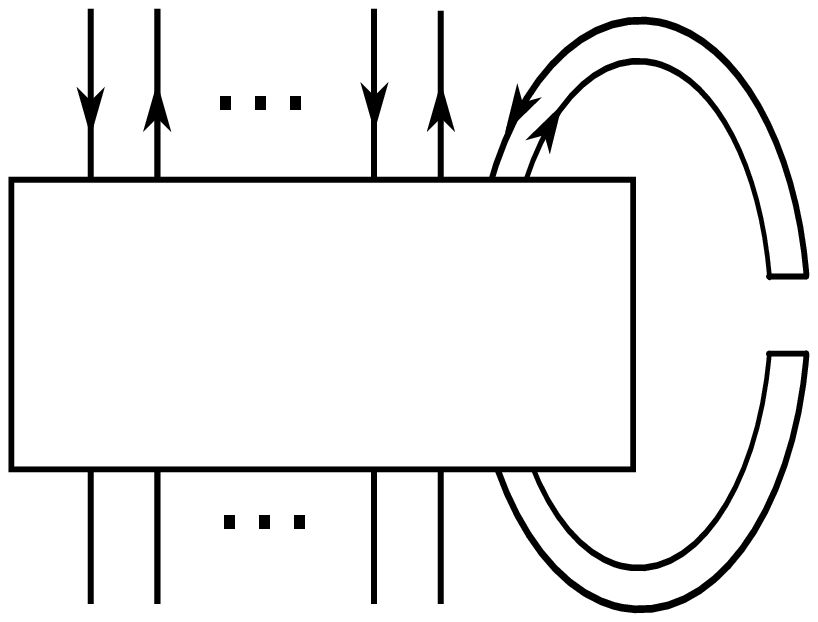}\label{fig:unknot2}}
\caption{}
\end{figure}
\end{proof}

There is a more naive perspective on fusions of NGHL's.

\begin{lem}\label{lem:fnghlare ribbon}  The set of $n$-component fusions of $k$ NGHL's is a subset of the set of links obtained from $n$-component ribbon links by adding $k$ generalized positive crossings.
\end{lem}
\begin{proof} An $n$-component  link $L$ which is a fusion of $k$ NGHL's is obtained by starting from a trivial link, adding $k$ generalized positive crossings (all of which are compatible with the technical restriction mentioned in Remark \ref{rmk:gpc nghl}), and then performing fusions to get to an $n$-component link. But since fusions correspond to bands whose cores are one-dimensional, and since the application site for a generalized positive crossing also is one-dimensional, these two operations commute.  That is to say, that $L$ may equally be obtained from the same trivial link by adding fusion bands (yielding a ribbon link), and then adding $k$ generalized positive crossings.
\end{proof}

Note that equality of these two sets does not necessarily hold.  If one views the addition of arbitrary GPC's to a ribbon link as addition of GPC's to an unlink prior to fusing, there's no \emph{a priori} guarantee that the latter GPC additions result in a disjoint union of NGHL's (or even a fusion of NGHL's).

Having made these observations, our characterizations follow easily.

\begin{thm}\label{thm:ZP}  Fix $k,n \in \mathbb{Z}_{>0}$. Let $L \subset S^3$ be an $n$-component link.  The following are equivalent:
\begin{enum}
\item  $L$ is slice in a (punctured) $\#_k \cp$.
\item  $L$ is concordant to a fusion of $k$ NGHL's.
\item $L$ is concordant to a link $L'$ which is obtained from a ribbon link by adding $k$ positive generalized crossings.
\end{enum}
\end{thm}

\begin{proof}  $\emph{(1)} \Rightarrow \emph{(2):}$ Suppose that $L$ is slice in $V:=\text{closure}(\#_k \cp \setminus B^4)$.  Let $L_1, L_2, \ldots, L_n$ denote the components of $L$ and $D_1, D_2, \ldots, D_n$ denote their respective slice disks in $V$.  Since each $D_i$ is trivial in $H_2(V,\partial V)$, it intersects $E_i$, the exceptional sphere of the $i^{th}$ $\mathbb{C}P(2)$ summand, in an even number of points, which cancel when counted with sign.  After removing open tubular neighborhoods of the exceptional spheres $E_1,...,E_k$, we are left with $n$ planar surfaces  embedded in a smooth manifold $W$ that is diffeomorphic to $B^4$ with $k$ open sub-balls deleted.  Choose $k-1$ arcs in $W$ that avoid the surfaces and connect the $3$-spheres $\partial N(E_i)$. Remove neighborhoods of these arcs from $W$. The result is a collection, $P$, of  $n$ planar surfaces properly embedded  in a manifold  diffeomorphic to $S^3 \times [0,1]$, forming a cobordism from $L$ to a link which is the disjoint union of $k$ NGHL's (using Remark~\ref{rem:Hopffib}). After an isotopy of $P$ we can assume that the induced height function  $P\hookrightarrow S^3 \times [0,1]\to [0,1]$ is a Morse function $f$ where $f(L)=1$, whose minima all occur at height $1/4$, maxima occur at height $3/4$, and where $f^{-1}(1/2)$ is an $n$-component link $L'$ which is concordant to $L$.   The link $L''=f^{-1}(1/4+\epsilon)$ is the disjoint union of $k$ NGHL's together with a trivial link created by the local minima.  Thus $L'$ is a fusion of $L''$. But the latter is a fusion of $k$ NGHL's by Lemma \ref{lem:nghl+triv}. Thus $L'$ is itself a fusion of $k$ NGHL's.

$\emph{(2)}\Rightarrow \emph{(3):}$   This follows from Lemma~\ref{lem:fnghlare ribbon}.

$\emph{(3)}\Rightarrow \emph{(1):}$ Any ribbon link is slice in $B^4$. Thus, by the last line in the proof of Corollary~\ref{cor:GPC}, $L'$ is slice in a (punctured) $\#_k \cp$. Thus $L$ is slice in a (punctured) $\#_k \cp$.
\end{proof}

\begin{cor}\label{cor:equivalencerel} $\widetilde{\mathcal{P}}_0$ is the smallest set that contains all trivial links and is closed under concordance and adding generalized positive crossings.
\end{cor}

%\begin{prop}
%Let $L \subset S^3$ be a (NAME?) link, and choose Seifert surfaces $F_r$ (resp. $F_b$) for the components $L_r$ (resp. $L_b$) of $L$ as in Remark \ref{rmk:SS}.  Let $N_b$ (resp. $N_r$) be the signed count of intersections between the bands of $F_b$ (resp. bands of $F_r$) and annuli of $F_r$ (resp. annuli of $F_b$).  Then the Sato-Levine invariant of $L$ is
%$$ \beta(L) = -\left( N_r + N_b \right)^2$$
%In particular, it vanishes exactly when $N_r = - N_b$.
%\end{prop}
%
%\begin{proof}
%Starting at one end of the red band, we successively modify intersections of the red band (resp. blue band) with the blue annulus (resp. red annulus) by tubing the annulus along the boundary of the band, as shown in (FIGURE), creating a series of nested tubes.  Now the red and blue surfaces intersect in a disjoint union of simple closed curves which can be retracted along the red band without passing one another to yield a collection of parallel copies of the core of the blue annulus.
%\end{proof}
\section{The Conway polynomial of a two-component link that is slice in $\mathbb{C}P^2$}\label{sec:cp2}

Let $L \subset S^3$ be a $k$-component link, and choose a connected Seifert surface $F\subset S^3$ for $L$.  Then if $V$ is the Seifert matrix for $F$ recall that the Conway polynomial $\nabla_L(z)$ of $L$ is obtained by substituting $z = x - x^{-1}$ into the expression
$$ \text{det}\left(xV - x^{-1}V^t\right)$$
Recall that in fact the Conway polynomial has the form
$$ \nabla_L(z) = z^{k-1}\left( a_0 + a_1 z^2 + a_2 z^4 + \ldots + a_k z^{2m} \right)$$
for some coefficients $a_i \in \mathbb{Z}$.

When computing Conway polynomials below, the following notational convention will be useful.  Let $A:=A(x)$ be a matrix whose entries are Laurent polynomials in $\mathbb{Z}[x,x^{-1}]$.  Then define a new matrix $\overline{A}:=A(-x^{-1})$.  Notice in particular that
\begin{enumerate}[(i)]
\item For any $(n \times m)$ matrices $A$ and $B$, $\overline{A+B} = \overline{A}+\overline{B}$
\item For any $(n\times m)$ matrix $A$ and $(m \times k)$ matrix $B$, $\overline{A\cdot B} = \overline{A}\cdot\overline{B}$
\item For any square matrix $A$, $\text{det}(\overline{A}) = \overline{\text{det}(A)}$
\item If $B$ is obtained from $A$ via some elementary row or column operation, then the result of performing the same operation on $\overline{A}$ is $\overline{B}$
\end{enumerate}

We develop the following obstruction to a two-component link being slice in a punctured (non-exotic) $\mathbb{C}P^2$.

\begin{thm}\label{thm:conway}
Let $L \subset S^3$ be a 2-component link which is slice in a punctured $\mathbb{C}P^2$, and suppose that the Conway polynomial $\nabla_L$ of $L$ is of the form
$$ \nabla_L(z) = z \left( a_k z^{2k} + a_{(k+1)} z^{2k+2} + \ldots + a_n z^{2n} \right) \quad \text{where} \quad a_k \neq 0.$$
Then $(-1)^k a_k \leq 0$.
\end{thm}

\begin{cor}\label{cor:conway}
Let $L \subset S^3$ be a 2-component link which is both slice in a punctured $\cp$ and slice in a punctured $\overline{\cp}$.  Then $\nabla_L(z) \equiv 0$.
\end{cor}

\begin{conj}\label{conj:conway}
The conclusion of Theorem \ref{thm:conway} holds whenever $L\in \PZ$ (though in general, the sign $(-1)^k$ should be replaced with $(-1)^{s(k,m)}$ for some function $s(k,m)$ depending on the leading degree $k$ and the number of components $m$).
\end{conj}

\begin{rmk}
Recall that the components of $0$-positive links must have linking number equal to zero.  As a result, $a_0 = 0$ and we recover that $\beta(L) = -a_1 \leq 0$.  The first author showed in \cite{C5} that if a 2-component link $L$ has pairwise linking zero and $\beta(L)=0$, then the coefficient of $z^5$ in $\nabla_L(z)$ is equal to $\alpha + \gamma - 2\delta$, where $\alpha:=\mubar_L(111122)$, $\gamma :=\mubar_L(112222)$, and $\delta:=\mubar_L(111222)/2$.  Theorem \ref{thm:conway} then implies that when such a link is slice in $\cp$,
$$\mubar_L(111222) \geq   \mubar_L(111122) +  \mubar_L(112222) = \overline{\beta}^2(L) + \beta^2(L)$$
Furthermore, if $L$ is slice in both of $\pm\cp$, $\mubar_L(111222) = \overline{\beta}^2(L) + \beta^2(L) = 0+0=0$.

Figure 4.6 of \cite{C5} exhibited a family of such two-component links realizing any integer values of the concordance invariants $\alpha$, $\gamma$, and $\delta$.  Let $L$ be the link from that family with $\alpha = \gamma = -1$ and $\delta = -2$, for instance.  Then Theorem \ref{thm:conway} implies that $L$ is not slice in a punctured $\cp$.  While Corollary \ref{cor:betan} implies that $L$ is not slice in a punctured $\#_k \overline{\cp}$ for any value of $k$, it gives no conclusion as to whether $L$ can be slice in a punctured $\#_k \cp$.
\end{rmk}

\begin{proof}[Proof of Theorem \ref{thm:conway}]
By Theorem \ref{thm:ZP}, we have that $L$ is concordant to some two-component link $L'$ which is a fusion of a 2-colored NGHL like the one depicted in Figure \ref{fig:nghl}.  While the Conway polynomial itself isn't a concordance invariant, Theorem 3.2 in \cite{C5} guarantees that the smallest degree non-vanishing coefficient (and its degree) is - so, it suffices to prove the theorem for the link $L'$.

Let the two colors on components arising in the construction of $L'$ be red and blue, and let $L_r$ and $L_b$ denote the red and blue components of $L'$, respectively.  Before fusing, for each $i$, choose a matching of the $2n$ red components of $\tld{L}$ in oppositely-oriented pairs and let $A^i_r$ denote the annulus cobounded by the $i^{th}$ pair ($1 \leq i \leq n$); similarly match the $2m$ blue components in pairs and define annuli $A^i_b$ for $1 \leq i \leq m$.  Abstractly, let the immersed surface $\tld{F}_r$ (resp. $\tld{F}_b$) be the union of the red annuli (resp. the blue annuli) with the red fusion bands (resp. the blue fusion bands).

We form an immersed connected surface $\tld{F}$ with boundary $L$ by tubing together the annuli $A^1_r$ and $A^1_b$.  After possibly deforming $\tld{F}$ (``sliding the feet'' of bands along other bands and pushing ribbon intersections along bands), we can assume that the following hold:
\begin{enumerate}[(i)]
\item All ribbon intersections occur on the annuli $A^1_r$ and $A^1_b$
\item One red (resp. blue) band fuses the two boundary components of $A^1_r$ (resp. $A^1_b$)
\item Each of the other $(2n-2)$ red bands (resp. $(2m-2)$ blue bands) fuses the same boundary component of $A^1_r$ (resp. $A^1_b$) to a boundary component of one of the other $(n-1)$ red annuli (resp. $(m-1)$ blue annuli)
\item Exactly one band is incident to each boundary component of each of $A^i_r$ and $A^i_b$ for $i > 1$
\end{enumerate}

We introduce some terminology related to a ribbon intersection between a band and an annulus:
\begin{enumerate}[(i)]
\item An intersection between a red band and $A_r^1$ is called \textbf{r-monochromatic}
\item An intersection between a red band and $A_b^1$ is called \textbf{r-polychromatic}
\item An intersection between a blue band and $A_b^1$ is called \textbf{b-monochromatic}
\item An intersection between a blue band and $A_r^1$ is called \textbf{b-polychromatic}
\end{enumerate}

\begin{figure}[h!]
\centering
\begin{minipage}[c]{.40\textwidth}
\includegraphics[height = 30mm]{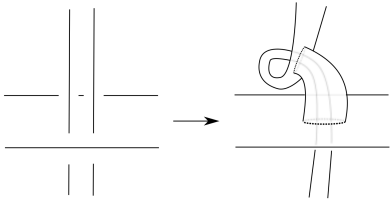}
\end{minipage}
\begin{minipage}[c]{.55\textwidth}
\caption{Resolving a ribbon intersection on an immersion of a connected surface increases the genus by one\label{fig:resolve}}
\end{minipage}
\end{figure}

Abstractly, $\tld{F}$ has two boundary components and is of genus $(m+n)$, where $m$ and $n$ are as above.  We can resolve each of the $k$ ribbon intersections as shown in Figure \ref{fig:resolve} to obtain an embedded surface $F$ of genus $g:=m+n+k$.  Label the red (resp. blue) bands $B_r^1, B_r^2, \ldots, B_r^{(2n-1)}$ (resp. $B_b^1, B_b^2, \ldots, B_b^{(2m-1)}$) in any order.  For $1 \leq i \leq (2n-1)$ and for $1 \leq j \leq n_i$, let $T_r^{i,j}$ denote the tube added while resolving the $j^{th}$ ribbon intersection involving the band $B_r^i$ (where the intersections are ordered sequentially along $B_r^i$ starting near $A_r^1$); analogously label tubes arising from resolving blue-band ribbon intersections as $T_b^{i,j}$ for $1 \leq i \leq (2m-1)$ and $1 \leq j \leq m_i$.  By abuse of terminology, we'll say (for instance) that ``$T_r^{i,j}$ is r-monochromatic'' if the ribbon intersection whose resolution produced $T_r^{i,j}$ was r-monochromatic.  Note that
$$ \displaystyle \sum_{i=1}^{2n-1} n_i + \sum_{i=1}^{2m-1} m_i = k$$

We now describe a set of $(2g+1)$ simple closed curves on $F$ representing a basis for $H_1(F;\mathbb{Z})$.  Let $a_r^i$ be the core of $A^i_r$ for $1 \leq i \leq n$ and let $a_b^i$ be the core of $A_b^i$ for $1 \leq i \leq m$.  Then let $c_r^1$ (resp. $c_b^1$) be the dual curve to $a_r^1$ (resp. $a_b^1$) which traverses the band fusing the two boundary components of $A^1_r$ (resp. $A^1_b$).  For $2 \leq i \leq n$, let $c_r^i$ be the dual curve to $a_b^i$ which travels along a band from $A^1_r$ to $A^i_r$ and returns to $A^1_r$ along another band; for $2 \leq i \leq m$, the curves $c_b^i$ are defined similarly.

\begin{figure}[h!]
\centering
\begin{minipage}[c]{.4\textwidth}
\labellist
\small
\pinlabel* {$e$} at 140 170
\pinlabel* {$a_r^i$} at 55 85
\pinlabel* {$a_b^j$} at 215 160
\pinlabel* {$c_b^k$} at 60 120
\pinlabel* {$q_b^{k,1}$} at 90 275
\pinlabel* {$q_b^{k,2}$} at 155 90
\pinlabel* {$p_b^{k,1}$} at 215 300
\pinlabel* {$p_b^{k,2}$} at 250 145
\pinlabel* {$A_b^j$} at 245 220
\pinlabel* {$A_r^i$} at 245 85
\pinlabel* {$B_b^k$} at 140 30
\endlabellist
\includegraphics[width = 70mm]{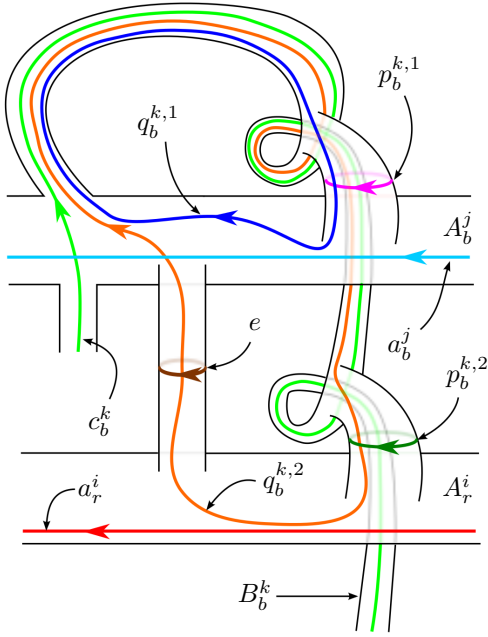}
\end{minipage}
\begin{minipage}[c]{.55\textwidth}
\caption{An example illustrating representatives for several members of our chosen basis for $H_1(F;\mathbb{Z})$\label{fig:basis}}
\end{minipage}
\end{figure}

Now for $1 \leq i \leq (2n-1)$ and for $1 \leq j \leq n_i$, let $p_r^{i,j}$ be a meridian curve of the tube $T_r^{i,j}$ and let $q_r^{i,j}$ be a curve dual to $p_r^{i,j}$ which is chosen according to the following rules:  
\begin{enumerate}[(i)]
\item If $T_r^{i,j}$ is r-monochromatic, then $q_r^{i,j}$ starts at a point $x$ near where $T_r^{i,j}$ attaches to $B_r^i$, travels longitudinally along $T_r^{i,j}$ to $A_r^1$, then along $A_r^1$ to the attachment point of $B_r^i$, and finally back along $B_r^i$ to $x$
\item If $T_r^{i,j}$ is r-polychromatic, then $q_r^{i,j}$ starts at a point $x$ near where $T_r^{i,j}$ attaches to $B_r^i$, travels longitudinally along $T_r^{i,j}$ to $A_b^1$, then along the tube to $A_r^1$, then along $A_r^1$ to the attachment point of $B_r^i$, and finally back along $B_r^i$ to $x$
\end{enumerate}

For $1 \leq i \leq (2m-1)$ and for $1 \leq j \leq m_i$, $p_b^{i,j}$ and $q_b^{i,j}$ are defined similarly.  Finally, let $e$ be a meridian of the tube connecting $A_r^1$ and $A_b^1$; this completes the basis for $H_1(F; \mathbb{Z})$.  Figure \ref{fig:basis} exhibits an example of a piece of such a surface $F$ with several of the curves we've described labelled.

Notice that the $(2g+1)$ curves that we've described can be chosen to be pairwise disjoint aside from the following exceptions:
\begin{enumerate}[(i)]
\item For each $1 \leq i \leq (2n-1)$ (resp. $1 \leq i \leq (2m-1)$) and $\leq j \leq n_i$ (resp. $1 \leq j \leq m_i$) , the curves $p_r^{i,j}$ and $q_r^{i,j}$ (resp. $p_b^{i,j}$ and $q_b^{i,j}$) intersect once
\item For each $1 \leq i \leq (2n-1)$ (resp. $1 \leq i \leq (2m-1)$), the curves $a_r^i$ and $c_r^i$ (resp. $a_b^i$ and $c_b^i$) intersect once
\item For each $1 \leq i \leq (2n-1)$ (resp. $1 \leq i \leq (2m-1)$) and $\leq j \leq n_i$ (resp. $1 \leq j \leq m_i$), if $T_r^{i,j}$ is r-polychromatic (resp. $T_b^{i,j}$ is b-polychromatic), then the curves $q_r^{i,j}$ (resp. $q_b^{i,j}$) and $e$ intersect once
\end{enumerate}
Linking among these curves is as follows:
\begin{enumerate}[(i)]
\item For any $1\leq i,j \leq (2n-1)$ and $1\leq k,l \leq (2m-1)$
$$\ell k(a_r^i, (a_r^j)^+) =\ell k(a_b^k, (a_b^l)^+) =\ell k(a_r^i, (a_b^k)^+) = \ell k(a_b^k, (a_r^i)^+)=-1$$
\item For any $1\leq i,j \leq (2n-1)$ and $1\leq k,l \leq (2m-1)$, each of $\ell k(c_r^i, (c_r^j)^+) = \ell k(c_r^j (c_r^i)^+)$, $\ell k(c_b^k, (c_b^l)^+) = \ell k(c_b^l (c_b^k)^+)$ ,and $\ell k(c_r^i, (c_b^k)^+) = \ell k(c_b^k, (c_r^i)^+)$ can be arbitrary.
\item For any $1 \leq i,j \leq (2n-1)$ and $1 \leq k,l \leq (2m-1)$, $\ell k(c_r^i, (a_r^j)^+)$ and $\ell k(c_b^k, (a_b^l)^+)$ can be arbitrary and
$$ 
\ell k(a_r^j, (c_r^i)^+) = \begin{cases}
	\ell k(c_r^i, (a_r^j)^+) + 1 & \text{if} \quad j =i\\
	\ell k(c_r^i, (a_r^j)^+) & \text{otherwise}
\end{cases}
\quad \text{and}
\quad \ell k(a_b^l, (c_b^k)^+) = \begin{cases}
	\ell k(c_b^j, (a_b^j)^+)+1 & \text{if} \quad j =i\\
	\ell k(c_b^j, (a_b^j)^+) & \text{ otherwise}
\end{cases}
$$
\item For an element $x$ in the basis and for any $1 \leq i \leq (2n-1)$ and $1 \leq j \leq n_i$,
\begin{equation}\label{eqn:plinking}
 \ell k\left( p_r^{i,j}, x^{+} \right) = \begin{cases}
\lambda_r^{i,j} & \text{if } x = q_r^{i,j}\\
1 & \text{if } x = q_r^{i,k}, k > j\\
\rho_r^{i} & \text{if } x = c_r^{i}\\
0 & \text{otherwise}
\end{cases}
\quad\text{and}\quad
\ell k\left( x,(p_r^{i,j})^{+} \right) = \begin{cases}
1-\lambda & \text{if } x = q_r^{i,j}\\
1 & \text{if } x = q_r^{i,k}, k > j\\
\rho_r^{i} & \text{if } x = c_r^{i}\\
0 & \text{otherwise}
\end{cases}
\end{equation}
where $\lambda_r^{i,j} \in \{ 0, 1\}$ depends on the sign of the ribbon intersection generating $T_r^{i,j}$ and $\rho_r^{i} \in \left\{ -1, 1 \right\}$ depends on the band $B_r^{i}$.  Similar statements hold for $ \ell k\left( p_b^{i,j}, x^{+} \right)$ and $\ell k\left( x,(p_b^{i,j})^{+} \right)$.
\item For an element $x$ in the basis, $\ell k \left( e,x^{+} \right)=0$.  Furthermore,
$$\ell k \left( x, e^+ \right) = \begin{cases}
1 & \text{if $x=q_r^{i,j}$ and $T_r^{i,j}$ is r-polychromatic}\\
-1& \text{if $x=q_b^{i,j}$ and $T_b^{i,j}$ is b-polychromatic}\\
0 & \text{otherwise}
\end{cases}
$$
\item For an element $x$ in the basis and for any $1 \leq i \leq (2n-1)$, $1 \leq k \leq (2m-1)$,  $1 \leq j \leq n_i$, and $1 \leq l \leq m_k$, $\ell k ( q_{r}^{i,j}, x^+ ) = \ell k ( x, ( q_r^{i,j})^+ )$ and $\ell k ( q_{b}^{k,l}, x^+ ) = \ell k ( x, ( q_b^{k,l})^+ )$ can be arbitrary as long as $x \notin \left\{ p_r^{k,l} \right\}_{k,l} \cup \left\{ p_b^{k,l} \right\}_{k,l} \cup \left\{ e \right\}$.
\end{enumerate}

Let $\bold{a}^r$ denote the list of $a^r_{i}$ ordered sequentially (and similarly for $\bold{a}^b$, $\bold{c}^r$, and $\bold{c}^b$.  Let $\bold{p}^r$ denote the list of $p^r_{i,j}$, ordered lexicographically (and similarly for $\bold{p}^b$, $\bold{q}^r$, and $\bold{q}^b$).  Then order our basis as $\bold{a}^r$, $\bold{a}^b$, $\bold{c}^r$, $\bold{c}^b$, $e$, $\bold{p}^r$, $\bold{p}^b$, $\bold{q}^r$, $\bold{q}^b$.  The Seifert matrix with respect to this ordered basis is then

$$
V = \begin{blockarray}{rcccccl}
	& (\bold{a}^r)^+,(\bold{a}^b)^+ & (\bold{c}^r)^+,(\bold{c}^b)^+ & e^+ & (\bold{p}^r)^+,(\bold{p}^b)^+ & (\bold{q}^r)^+,(\bold{q}^b)^+ & \\
	\begin{block}{r(c|c|c|c|c)l}
		\bold{a}^r, \bold{a}^b & -1 & A_2  &0 & 0 & B^t & \boldsymbol{\Bigg\}} \hspace{2mm} m+n\\ \cline{2-6}
		\bold{c}^r, \bold{c}^b & A_1 & C & 0 &  D^t & F^t &  \boldsymbol{\Bigg\}} \hspace{2mm} m+n\\ \cline{2-6}
		e & 0 & 0 & 0 & 0 & E_2 &\\ \cline{2-6}
		\bold{p}^r, \bold{p}^b & 0 & D & 0  & 0 & P_2 & \boldsymbol{\Bigg\}} \hspace{2mm} k\\ \cline{2-6}
		\bold{q}^r, \bold{q}^b & B & F & E_1 & P_1 & Q &  \boldsymbol{\Bigg\}} \hspace{2mm} k\\
	\end{block}
	&\underbrace{\qquad\qquad}_{m+n}&\underbrace{\qquad\qquad}_{m+n}&&\underbrace{\qquad\qquad}_{k}&\underbrace{\qquad\qquad}_{k}&
\end{blockarray} \in \mathbb{Z}^{(2g+1) \times (2g+1)}
$$
where the ``$-1$'' indicates a matrix with all entries equal to $-1$.  Notice also that $C$ and $Q$ are symmetric matrices.  Now letting $A:=xA_1 - x^{-1}A_2^t$, $E:=xE_1-x^{-1}E_2^t$, and $P:=xP_1-x^{-1}P_2^t$, we have that $M:=xV - x^{-1}V^t$ is given by

$$
M  = \begin{blockarray}{cccccl}
	\begin{block}{(c|c|c|c|c)l}
		-Z & \overline{A}^t  &0 & 0 & zB^t & \boldsymbol{\Bigg\}} \hspace{2mm} m+n\\ \cline{1-5}
		A & zC & 0 &  zD^t & zF^t &  \boldsymbol{\Bigg\}} \hspace{2mm} m+n\\ \cline{1-5}
		0 & 0 & 0 & 0 & \overline{E}^t &\\ \cline{1-5}
		0 & zD & 0  & 0 & \overline{P}^t & \boldsymbol{\Bigg\}} \hspace{2mm} k\\ \cline{1-5}
		zB & zF & E & P & zQ &  \boldsymbol{\Bigg\}} \hspace{2mm} k\\
	\end{block}
	\underbrace{\qquad\qquad}_{m+n}&\underbrace{\qquad\qquad}_{m+n}&&\underbrace{\qquad\qquad}_{k}&\underbrace{\qquad\qquad}_{k}&
\end{blockarray}
$$
where every entry of $Z$ is equal to $z = x - x^{-1}$.  Notice also that equation \ref{eqn:plinking} implies that $P$ is a lower-triangular matrix whose diagonal entries are $x_1, x_2, \ldots, x_k$ with $x_i \in \left\{ x, -1/x \right\}$.

Notice that
$$ \text{det}\begin{pmatrix}
0 & \overline{P}^t \\
P & zQ
\end{pmatrix} = \text{det}(-P \cdot \overline{P}^t) = \displaystyle \prod_{i = 1}^k (-x_i)\displaystyle \prod_{i = 1}^k (-\frac{1}{x_i}) = 1
$$

Recall that if a matrix $N$ has block form
$$ N = \begin{pmatrix}
U & W\\
X & Y
\end{pmatrix}
$$
where $U$ and $Y$ are square matrices and $Y$ is non-singular, then
$$\text{det}(N) = \text{det}(Y)\cdot\text{det}(U-WY^{-1}X)$$
As a result,

\begin{align*}
\text{det}(M) &= \text{det}\left( \begin{pmatrix}
	-Z & \overline{A}^t & 0\\
	A & zC & 0\\
	0 & 0 & 0
\end{pmatrix}
-\begin{pmatrix}
	0 & zB^t\\
	zD^t & zF^t\\
	0 & \overline{E}^t
\end{pmatrix}
\cdot \begin{pmatrix}
	0 & \overline{P}^t\\
	P & zQ
\end{pmatrix}^{-1}
\cdot \begin{pmatrix}
	0 & zD & 0\\
	zB & zF & E
\end{pmatrix}
\right)\\
&=\text{det} \left( \begin{pmatrix}
	-Z & \overline{A}^t & 0\\
	A & zC & 0\\
	0 & 0 & 0
\end{pmatrix}
-\begin{pmatrix}
	0 & zB^t\\
	zD^t & zF^t\\
	0 & \overline{E}^t
\end{pmatrix}
\cdot \begin{pmatrix}
	-zP^{-1} Q (\overline{P}^{-1})^t & P^{-1}\\
	(\overline{P}^{-1})^t & 0
\end{pmatrix}
\cdot \begin{pmatrix}
	0 & zD & 0\\
	zB & zF & E
\end{pmatrix}
\right)\\
&=\text{det} \begin{pmatrix}
-Z & \overline{S}^t & 0\\
S & R & zT\\
0 & z\overline{T}^t & 0
\end{pmatrix}
\end{align*}
where
\begin{gather}\label{eqn:mat}
\begin{aligned}
R &:= zC + z^2\left( zD^tP^{-1} Q (\overline{P}^{-1})^t - F^t (\overline{P}^{-1})^t \right)D - z^2D^t P^{-1} F,\\
S &:=A - z^2D^t P^{-1}B, \quad \text{and} \quad T:=-D^t P^{-1} E
\end{aligned}
\end{gather}
Notice that $S$ and $R$ are $(m+n) \times (m+n)$ matrices and $T$ is a $(m+n)\times 1$ matrix.  Furthermore, a careful examination of equation \ref{eqn:mat} reveals that the entries of $S$ (resp. $T$) are elements of $\mathbb{Z}[x,x^{-1}]$ with terms of only odd (resp. even) degree.

Let $Z_0$ be a $(m+n) \times (m+n)$ with $z$ as the $(1,1)$ entry and all other entries equal to zero.  Let $S_0$ be the result of subtracting the first column of $S$ from every other column, and let $S_1$ be the matrix obtained from $S_0$ by deleting the first column.  Then
\begin{align*}
\text{det}\begin{pmatrix}
	-Z & \overline{S}^t & 0\\
	S & R & zT\\
	0 & z\overline{T}^t & 0
\end{pmatrix}&=
\text{det}\begin{pmatrix}
	-Z_0 & \overline{S}^t_0 & 0\\
	S_0 & R & zT\\
	0 & z\overline{T}^t & 0
\end{pmatrix}
=-z \text{det} \begin{pmatrix}
	0 & \overline{S}^t_1 & 0\\
	S_1 & R & zT\\
	0 & z\overline{T}^t & 0
\end{pmatrix} + \begin{pmatrix}
	0 & \overline{S}^t_0 & 0\\
	S_0 & R & zT\\
	0 & z\overline{T}^t & 0
\end{pmatrix}\\
&= -z \text{det} \begin{pmatrix}
	0 & \overline{S}^t_1 & 0\\
	S_1 & R & zT\\
	0 & z\overline{T}^t & 0
\end{pmatrix}
= -z^3  \text{det} \begin{pmatrix}
0 & 0 & \overline{S}^t_1\\
0& 0&\overline{T}^t\\
S_1&T& R
\end{pmatrix}\\
& = -z^3(-1)^{(m+n)}\text{det} \begin{pmatrix}
S_1 &T
\end{pmatrix} \text{det} \begin{pmatrix}
\overline{S}_1 & \overline{T}
\end{pmatrix}\\
&= z^3 \text{det} \begin{pmatrix}
xS_1 & T
\end{pmatrix} \text{det}\overline{\begin{pmatrix}
xS_1 & T
\end{pmatrix}}
\end{align*}
Thus there exists a polynomial $f \in \mathbb{Z}[t]$ such that
$$ \text{det}(M) = z^3 f(x^2) f(x^{-2}).$$
The result now follows from Proposition \ref{prop:poly1} in Appendix \ref{sec:poly} below.
\end{proof}

\section{Signature obstructions}\label{sec:signatures}
Let $L \subset S^3$ be a link with $m$ components, choose a connected Seifert surface $F \subset S^3$ for $L$ and let $V$ be a Seifert matrix for $F$ associated to $H_1(F)$.  Given a norm $1$ complex number $\omega$, recall that the Levine-Tristram $\omega$-signature of $L$, denoted $\sigma_L(\omega)$, is defined to be the signature of the Hermitian form
$$ (1-\omega)V + (1-\overline{\omega})V^t$$
However, for $\omega$ equal to a zero of the determinant of this matrix, we redefine $\sigma_L(\omega)$ to be the average of the two limits $\displaystyle\lim_{\alpha\rightarrow \omega^{\pm}} \sigma_L(\alpha)$.  The resulting function, $\sigma_L:S^1\to \Z$, we shall call the \textbf{Levine-Tristram signature function} of $L$. Note that when $m=1$ (i.e. $L$ is a knot), the determinant vanishes for at most finitely-many $\omega \in S^1$.  Although this might not be true when $m > 1$, the function $\sigma_L$ is still locally constant away from a finite number of points in $S^1$.  This function is a concordance invariant. If $p$ is a prime the signatures corresponding to $\omega^j$ where $\omega=\exp(\frac{2\pi i}{p^r})$ and $0<j< p^r$ are called the $p^r$-signatures of $L$. In particular the case $p=2,r=j=1=-\omega$ is the signature of $V+V^T$, the \textbf{classical signature of $L$}.

The following generalizes  ~\cite[Prop. 4.1]{CHH}. The proof is almost identical.

\begin{thm}\label{Theorem:signatures} If $L\in \mathcal{P}_0$,  then the Levine-Tristram signature function of $L$ is non-positive.

\end{thm}

\begin{ex}\label{ex:sig}
Signatures can sometimes obstruct zero-positivity for links which elude our other methods.  Let $L$ be the two-component link obtained by taking positive untwisted Whitehead doubles of both components of a Hopf link, as shown in Figure \ref{fig:wh}.  Since $L$ is a boundary link, all of its Milnor's invariants (and its Conway polynomial) vanish.  Notice that $L$ is $\geq_0$ a totally split link and has unknotted components - however, the reader can verify that the ordinary signature $\sigma_L(-1)$ is positive, and so $L$ is not a zero-positive link.

\begin{figure}[h!]
\centering
\begin{minipage}[c]{.35\textwidth}
\includegraphics[height = 30mm]{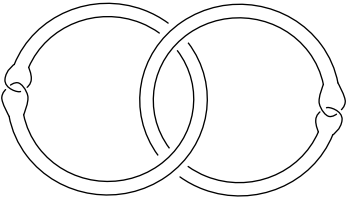}
\end{minipage}
\begin{minipage}[c]{.6\textwidth}
\caption{The classical signature obstructs this link's membership in $\PZ$ \label{fig:wh}}
\end{minipage}
\end{figure}
\end{ex}

\begin{proof}[Proof of Theorem~\ref{Theorem:signatures}] We follow the proof of ~\cite[Proposition 4.1]{CHH}. Since $L\in \mathcal{P}_0$, the components $L_i, 1\leq i\leq m$ of $L$ bound disjoint slice disks $\Delta_i$ in a manifold $V$ as in Definition~\ref{def:ZP}.  Let $d=p^r$ be a prime power. The kernel of the map to $Z_d$ that sends each meridian to $1$ corresponds to a $d$-fold covering space of the link exterior.  Let $\Sigma$ denote the corresponding $d$-fold cyclic cover of $S^3$ branched over $L$. Let $\Delta$ denote the disjoint union of the $\Delta_i$. Note that $H_1(V-\Delta)\cong\Z^m$, generated by the meridians, while  $H_2(V-\Delta)\cong H_2(V)$ and $H_3(V-\Delta)\cong H_3(V)=0$. The $d$-fold cyclic cover of $V$ branched over $\Delta$ will be denoted  $\widetilde{V}$. 

Let $(B^4,F_L)$ be the $4$-ball together with a connected Seifert surface for $L$ pushed into its interior. Let $(Y,F)=(V,\Delta)\cup (-B^4,-F_L)$, let $\widetilde{W}$ denote the branched cover of $(B^4,F_L)$, and let $\widetilde{Y}$ be the  branched cover of $(Y,F)$.  Since $H_1(Y)=0$, $H_1(B^4)=0$, $H_1(Y-F)\cong\Z\cong H_1(B^4-F_L)$,  it follows from the proof of ~\cite[Lemma 4.2]{CG1}, applied to $B^4$ and to $Y$, that $H_1(\widetilde{Y};\Z_p)=0=H_1(\widetilde{W};\Z_p)$. Thus $\beta_1(\widetilde{Y})=0=\beta_1(\widetilde{W})$.

Let $H_i(\widetilde{Y},j;\mathbb{C})$, $0\leq j<d$, denote the $\exp(\frac{2\pi i j}{d})$-eigenspace for the action of  a generator, $\tau$, of the group of deck transformations on $H_i(\widetilde{Y};\mathbb{C})$; let $\beta_i(\widetilde{Y},j)$ denote its rank, and let $\chi(\widetilde{Y},j)$ denote the alternating sum of these ranks (similarly for $\widetilde{V}$ and $\widetilde{W}$). Let $\sigma(\widetilde{Y},j)$ denote the signature of the eigenspaces of $\tau_*$ acting on $H_2(\widetilde{Y};\mathbb{C})$ (similarly for $\widetilde{V}$ and $\widetilde{W}$). By  ~\cite{Rok1}\cite[Lemma 2.1]{CG1},
$$
\sigma(\widetilde{Y},j)=\sigma(Y),
$$
which yields
$$
\sigma(\widetilde{V},j)-\sigma(\widetilde{W},j)=\sigma(V)-\sigma(B^4).
$$
Since the intersection form of $V$ is, by assumption, positive definite, we have

\begin{equation}\label{eq:sigs1}
\sigma(\widetilde{W},j)=\sigma(\widetilde{V},j)-\beta_2(V).
\end{equation}

By ~\cite[Proposition 1.1]{Gi5}, 
\begin{equation}\label{eq:eulerbar}
\chi(V-\Delta)=\chi(\widetilde{V}-\widetilde{\Delta},j).
\end{equation}
Now assume that $j>0$. We will evaluate the Betti numbers for each side of this equation. First we have already established that $\beta_0(V-\Delta)=1$, $\beta_1(V-\Delta)=m$, $\beta_2(V-\Delta)=\beta_2(V)$ and $\beta_3(V-\Delta)=0$.

Since $\tau$ acts by the identity on $H_0(\widetilde{V}-\widetilde{\Delta})$, since $j\neq 0$, $\beta_0(\widetilde{V}-\widetilde{\Delta},j)=0$. The inclusion-induced map on rational homology
$$
H_1(\widetilde{V}-\widetilde{\Delta},j)\to H_1(\widetilde{V},j)
$$
is surjective and its kernel is generated by the lifts of the $m$ meridians. Since $\tau$ acts by the identity on the first homology of these meridians, this epimorphism is an isomorphism when $j\neq 0$. Thus $\beta_1(\widetilde{V}-\widetilde{\Delta},j)=\beta_1(\widetilde{V},j)$. Now consider the Mayer-Vietoris sequence with $\Q$-coefficients:
\begin{equation}\label{eq:MayerV}
H_1(\widetilde{\Sigma})\to H_1(\widetilde{V})\oplus H_1(\widetilde{W})\to H_1(\widetilde{Y}).
\end{equation}
Since $\beta_1(\widetilde{W})=\beta_1(\widetilde{Y})=0$, the first map is surjective and $H_1(\widetilde{V},\widetilde{\Sigma})=0$. It follows from Lefshetz duality that $\beta_3(\widetilde{V})=0$. 

The set of meridians of the $\Delta_i$ is linearly independent in $H_1(V-\Delta;\Q)$. Their inverse images form the set of meridians of $\widetilde{\Delta}$ and thus the latter set is linearly independent in $H_1(\widetilde{V}-\widetilde{\Delta};\Q)$.  Since $\widetilde{V}$ is obtained from $\widetilde{V}-\widetilde{\Delta}$ by adding $2$-handles along homologically independent loops,
$$
\beta_2(\widetilde{V}-\widetilde{\Delta},j)=\beta_2(\widetilde{V},j).
$$
and
$$
\beta_3(\widetilde{V}-\widetilde{\Delta},j)=\beta_3(\widetilde{V},j)=0. 
$$

Thus, collecting all our information,  equation ~(\ref{eq:eulerbar}) becomes
\begin{equation}\label{eq:rankH2}
1-m+\beta_2(V)=-\beta_1(\widetilde{V},j)+\beta_2(\widetilde{V},j).
\end{equation}
Combining this with equation ~(\ref{eq:sigs1}) ,   we have
\begin{equation}\label{eq:sigs2}
\sigma(\widetilde{W},j)=\sigma(\widetilde{V},j)-\beta_2(\widetilde{V},j)+\beta_1(\widetilde{V},j)+1-m.
\end{equation}
The term $\sigma(\widetilde{V},j)-\beta_2(\widetilde{V},j)$ is always non-positive. By equation ~(\ref{eq:MayerV}), $\beta_1(\widetilde{V},j)$ is at most $\beta_1(\widetilde{\Sigma},j)$. Thus we have
\begin{equation}\label{eq:sigs3}
\sigma(\widetilde{W},j)\leq \beta_1(\widetilde{\Sigma},j)+1-m.
\end{equation}
It is known that $\beta_1(\widetilde{\Sigma},j)$ equals the $\omega$-nullity of $L$, $\eta_\omega(L)$, where $\omega=\exp(\frac{2\pi i}{d})$; and that this is bounded above by $m-1$ ~\cite[p.213]{Kauf78}\cite[Corollary 2.24]{Trist}. Thus $\sigma(\widetilde{W},j)$ is non-positive. But it is known that if $j>0$, $\sigma(\widetilde{W},j)$ is a $p^r$-signature of $L$ ~\cite{Vi1}\cite[Chapter 12]{Gor1};  specifically
$$
\sigma(\widetilde{W},j)=\sigma_{\omega^j}(L).
$$
Since the roots of unity, as $p^r$ varies, are dense in the circle, this implies that the entire signature function of $L$ is non-positive.
\end{proof}

\section{Using Rasmussen's s-invariant to obstruct membership in $\PZ$}\label{sec:s}

It is sometimes possible to obstruct the membership of a link in $\PZ$ by obstructing membership of the knots which are its components - many obstructions for knots are studied in \cite{CHH}.  However, this strategy fails when the link's components are slice knots, for example.  Notice that $\PZ$ is closed under taking fusions - in particular, one may in principle obstruct the membership of a link $L$ in $\PZ$ by fusing $L$ into a knot $K_L$ (in one of many possible ways) and then using knot concordance invariants to obstruct membership of $K_L$.  The following was proved by Kronheimer and Mrowka (rephrased to fit out notation here):
%One such invariant is Rasmussen's \textbf{s-invariant} \cite{Ras1}, which was extracted from Lee's filtered version of the Khovanov complex \cite{Lee} and descends to a group homomorphism  $s: \mathcal{C} \rightarrow 2\mathbb{Z}$ (here $\mathcal{C}$ denotes the smooth knot concordance group).  In fact, $|s(K)/2|$ gives a lower bound on the slice genus of $K$, which is in fact sharp when $K$ is a positive knot.  Rasmussen used this fact to give the first purely combinatorial proof of Milnor's conjecture about the genus of torus knots - a fact which was first proved using gauge theory in \cite{KrMr2}.
\begin{thm}[Corollary 1.1 from \cite{KrMr1}]\label{thm:km}
If a knot $K$ is in $\PZ$, then $s(K) \geq 0$.
\end{thm}

\begin{ex}\label{ex:bor}
Recall that the link in Example \ref{ex:sig} above was obstructed from membership in $\PZ$ by its signature function.  Let $L$ denote the three-component link exhibited in Figure \ref{fig:bor}, which is obtained by negative-Whitehead-doubling all components of the Borromean link.  This boundary link not only has trivial components, but also has vanishing signature function - thus eluding the methods provided above for obstructing membership in $\PZ$.

Letting $K_L$ be the knot obtained by performing the fusion indicated in Figure \ref{fig:bor}, we verified via computer that $s(K_L) = -2$ - as a result, $L \notin \PZ$.  This calculation was done using the function \texttt{UniversalKh} \cite{UKh}, a component of the package \texttt{KnotTheory`} that makes use of the program \texttt{JavaKh}; the reader should be warned that much of the mathematics underlying the function \texttt{UniversalKh} and its relationship to $s(K)$ is not in print.

Notice also that one can obtain the unlink by adding a GPC to $L$, so that $-L \in \PZ$.  It isn't known whether the link $L$ is topologically slice, but A. Levine \cite{LeAd} used the Ozsv\'ath-Szab\'o $\tau$-invariant from Knot Floer homology \cite{OzSz2} to show that $L$ isn't smoothly slice.

\begin{figure}[h!]
\centering
\begin{minipage}[c]{.25\textwidth}
\includegraphics[height = 40mm]{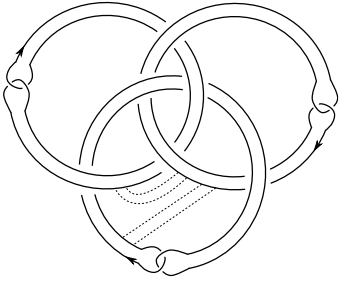}
\end{minipage}
\begin{minipage}[c]{.70\textwidth}
\caption{The $s$-invariant of the knot obtained by performing the indicated fusion obstructs this link's membership in $\PZ$ \label{fig:bor}}
\end{minipage}
\end{figure}
\end{ex}

Many other examples which elude the above classical obstructions can be obtained via Bing doubling, as the Bing double $B(K)$ of a knot $K \subset S^3$ is a boundary link with vanishing Levine-Tristram signature function and unknotted components.  In \cite{Ci}, Cimasoni observed that when $TB(K) \geq 0$, $s(B(K)) = 1$ and so $B(K)$ is not smoothly slice (here $s$ denotes the extension of Rasmussen's invariants to links described by Beliakova and Wehrli in \cite{BelW}) - we mimic his approach here.  A particular fusion of $B(K)$ yields the negative untwisted Whitehead double $Wh^{-}(K)$ of $K$.  Theorem 2 of \cite{LiN2} implies that if the Thurston-Bennequin invariant $TB(K)$ of $K$ is non-negative, then $s(Wh^{+}(K)) = 2$.  Therefore, choosing $K$ with $TB(-K) \geq 0$, one can see that $s(Wh^{-}(K))=-2$ and thus $B(K) \notin \PZ$.
%\begin{rmk}\label{rmk:slinks}
%Pardon \cite{Par} defined an extension of Rasmussen's invariant to links, which takes the form of a function $d_L:\mathbb{Z} \times \mathbb{Z} \rightarrow \mathbb{Z}_{\geq 0}$, where $d_L(i,j)$ is defined to be the dimension of the associated graded piece $\left( \text{Kh}_{\text{Lee}}^{h} \right)^s / \left( \text{Kh}_{\text{Lee}}^{h} \right)^{s+1}$ of Lee's filtered Khovanov homology (here $h$ denotes the homological grading and $s$ the filtration).  Recall that $\text{Kh}_{\text{Lee}}(L) \cong \mathbb{Q}^{2^m}$, where $m$ is the number of components of $L$.  Pardon's link concordance invariant $d_L$ is stronger than the version of $s$ for links given in \cite{BelW}, and one might expect that when $L \in \PZ$, $d_L$ has a special form, perhaps relative to the invariant $d_{U_m}$ corresponding to the $m$-component trivial link $U_m$.  Recall from \cite{Par} that
%$$ \displaystyle \sum_{h,s \in \mathbb{Z}} d_{U_m} (h,s) t^h q^s = \left( q + q^{-1} \right)^m$$
%\end{rmk}
\appendix
\section{Some particular Laurent polynomials}\label{sec:poly}
The goal of this section is to prove the following fact, which in turn completes the proof of Theorem \ref{thm:conway} above.  Recall that we let $x$ be a formal variable and define $z:=x - x^{-1}$.

\begin{prop}\label{prop:poly1}
Let $f \in \mathbb{Z}[t]$ be non-zero.  Then there are integers $a \neq 0$ and $k \geq 0$ and some polynomial $g \in \mathbb{Z}[t]$ such that 
$$ f(x^2)f(x^{-2}) =  (-1)^k a^2 z^{2k} + z^{2(k+1)}\left( g(z^2) \right).$$
In particular, $f(x^2)f(x^{-2})$ is a polynomial in $z^2$ and the sign of the coefficient of its lowest degree non-vanishing term is $(-1)^k$, where $k$ is half the degree of that term.
\end{prop}
We'll need several lemmas in order to prove Proposition \ref{prop:poly1}.
\begin{lem}\label{lem:poly1}
	For each integer $k \geq 0$, $x^{2k} + x^{-2k}$ is a polynomial in $z^2$ with constant coefficient equal to 2.
\end{lem}
\begin{proof}
Notice that for $m \geq 1$,
	$$ x^{2(m+1)} + \frac{1}{x^{2(m+1)}} = \left( x^{2m} + \frac{1}{x^{2m}} \right) \left( x^2 + \frac{1}{x^2}\right) - \left( x^{2(m-1)} + \frac{1}{x^{2(m-1)}} \right).$$
Considering that $x^0 + \frac{1}{x^0} = 2$ and $x^2 + \frac{1}{x^2} = z^2 + 2$, the result follows by strong induction.
\end{proof}
\begin{lem}\label{lem:poly2}
Let $m \in \mathbb{Z}_{\geq 0}$ and let $f \in \mathbb{Z}[t]$ be a polynomial given by
$$
f(t) = \displaystyle \sum_{i=1}^m \alpha_i t^{(m-i)}
$$
Then the Laurent polynomial $f(x^2)f(x^{-2}) \in \mathbb{Z}[x,x^{-1}]$ is equal to a polynomial in $z^2$ with constant coefficient given by
$$ \left( \alpha_1 + \alpha_2 + \ldots + \alpha_m \right)^2.$$
\end{lem}

\begin{proof}
Notice that
\begin{align*}
	f(x^2)f(x^{-2}) &= \sum_{i=1}^{m} \alpha_i^2 + \sum_{j = 1}^{m-1}\left( x^{2j} + \frac{1}{x^{2j}} \right)\left( \sum_{k=1}^{m-j} \alpha_{k}\alpha_{k+j} \right)\\
			&= \sum_{i=1}^{m} \alpha_i^2 + 2 \sum_{i \neq j} \alpha_i \alpha_j + z^2 g(z^2) = \left( \alpha_1 + \alpha_2 + \ldots + \alpha_m \right)^2 + z^2 g(z^2),
\end{align*}
where $g \in \mathbb{Z}[t]$ is a polynomial.  The first equality follows from the definition of $f$ and the second equality follows from Lemma \ref{lem:poly1}.
\end{proof}

\begin{lem}\label{lem:poly3}
Let $m \geq 1$, choose $\alpha_i \in \mathbb{Z}$ for $i = 1, \ldots, m$, and let $\beta_i := \alpha_1 + \ldots + \alpha_i$ for each $i$ with $1 \leq i \leq m$.  Let the polynomials $f,g \in \mathbb{Z}[t]$ be given by
\begin{gather*}
	\begin{aligned}
		f(t) &:=  \alpha_{m} + \alpha_{(m-1)} t + \alpha_{(m-2)} t^2 + \ldots + \alpha_1 t^{(m-1)} \quad \text{and}\\
		g(t) &:= \beta_{(m-1)} + \beta_{(m-2)} t + \beta_{(m-3)} t^2 + \ldots + \beta_1 t^{(m-2)}
	\end{aligned}
\end{gather*}
Then if $\beta_{m} = 0$, $f(x^2)f(x^{-2}) = -z^2 g(x^2) g(x^{-2})$.
\end{lem}
\begin{proof}
For each $i$ with $2 \leq i \leq m$, let the polynomials $f_i$ and $g_i$ be given by
\begin{gather*}
	\begin{aligned}
		f_i(t) &:=  -\beta_{(i-1)} + \alpha_{(i-1)} t + \alpha_{(i-2)} t^2 + \ldots + \alpha_1 t^{(i-1)} \quad \text{and}\\
		g_i(t)&:= \beta_{(i-1)} + \beta_{(i-2)}t + \beta_{(i-3)} t^2 + \ldots + \beta_1 t^{(i-2)}
	\end{aligned}
\end{gather*}
In particular, $f_m(t) = f(t)$ and $g_m(t) = g(t)$.  Now notice that for each $i$,
\begin{gather}
	\begin{aligned}
		f_i(x^2) &= x^2 f_{(i-1)}(x^2) + xz \beta_{(i-1)}, \quad
		&f_i(x^{-2}) &= \frac{f_{(i-1)}(x^{-2})}{x^2} - \frac{z}{x}\beta_{(i-1)}\\
		g_i(x^2) &= x^2 g_{(i-1)}(x^2) + \beta_{(i-1)}, \quad
		&g_i(x^{-2}) &= \frac{g_{(i-1)}(x^{-2})}{x^2} + \beta_{(i-1)}.\label{eqn:rel1}
	\end{aligned}
\end{gather}
Now fix some $i$.

\noindent \emph{Claim:} For each $j$ with $2 \leq j \leq i$,
\begin{equation}\label{eqn:rel2}
f_j(x^2) x^{2(i-j)+1} - \frac{f_{j}(x^{-2})}{x^{2(i-j)+1}} = z \left( g_j(x^2) x^{2(i-j) + 2} + \frac{g_j(x^{-2})}{x^{2(i-j) + 2}}  \right)
\end{equation}
The claim can be proved by induction on $j$.  The $j = 2$ case can be verified directly, and for $2 \leq j < i$, the $(j) \implies (j+1)$ inductive step follows easily from the relations appearing in equation \ref{eqn:rel1}.

Setting $j = i$ in equation \ref{eqn:rel2} provides that
\begin{equation}\label{eqn:rel3}
x f_i(x^2) - \frac{f_i(x^{-2})}{x} = z \left( x^2 g_i(x^2) + \frac{g_i(x^{-2})}{x^2} \right).
\end{equation}

\noindent\emph{Claim:} For each $i$ with $2 \leq i \leq m$, $f_i (x^2) f_i(x^{-2}) = -z^2 g_i(x^2) g_i(x^{-2})$

We proceed by induction on $i$.  It can be verified directly that $f_2(x^2) f_2(x^{-2}) = -z^2 g_2(x^2) g_2(x^{-2})$.  Equations \ref{eqn:rel1} and \ref{eqn:rel3} imply that when $2 \leq i \leq m-1$,
\begin{align*}
	g_{(i+1)}(x^2) g_{(i+1)}(x^{-2}) &= g_i(x^2) g_i(x^{-2}) + \beta_i\left( g_i x^2 + \frac{g_i(x^{-2})}{x^2}  \right) + \beta_i^2, \quad \text{and so}\\
	f_{(i+1)}(x^2) f_{(i+1)}(x^{-2}) &= f_i(x^2) f_i(x^{-2}) - z\beta_i \left( x f_i (x^2) - \frac{f_i(x^{-2})}{x} \right) - z^2 \beta_i^2\\
			&=f_i(x^2) f_i(x^{-2}) - z^2 \left( g_{(i+1)}(x^2) g_{(i+1)}(x^{-2}) - g_i(x^2) g_i(x^{-2}) \right).
\end{align*}
The above provides the $(i) \implies (i+1)$ induction step and the claim is proved.
\end{proof}

\begin{proof}[Proof of Proposition \ref{prop:poly1}]
Let $F(z):=f(x^2)f(x^{-2})$.  Suppose that the smallest degree of a non-vanishing term in $F(z)$ is $2k$.  For $1 \leq i, j \leq m$, define the numbers $\alpha_{i,j}$ recursively via the rule
$$
\alpha_{0,j}:= \alpha_{j} \quad \text{and} \quad \alpha_{i,j} := \sum_{l=1}^{j} \alpha_{(i-1),l}.$$
Then for each $i$ with $0 \leq i \leq m$, define the polynomial $f_i$ by
$$ f_i(t) = \displaystyle \sum_{j = 1}^{m-i} \alpha_{i,j}t^{(m-i-j)}.$$

\noindent\emph{Claim:} $f(x^2)f(x^{-2}) = (-z^2)^k f_k(x^2)f_k(x^{-2})$

Lemma \ref{lem:poly2} tells us that for each $i$, $f_i(x^2)f_i(x^{-2})$ is a polynomial in $z^2$ with constant coefficient equal to $\left( \alpha_{i,1} + \alpha_{i,2} + \ldots + \alpha_{i,(m-i)} \right)^2$; we proceed by induction on $i$.  Notice first that $f_0(x^2) = f(x^2)$.  Assume that for some $0 \leq i < k$, $f(x^2)f(x^{-2}) = (-z^2)^i f_i(x^2)f_i(x^{-2})$.  The coefficient of $z^{2i}$ in $f(x^2)f(x^{-2})$ vanishes by assumption, and so
$$ \alpha_{i,(m-i)} = -\left( \alpha_{i,1} + \alpha_{i,2} + \ldots + \alpha_{i, (m-i-1)} \right) = - \alpha_{(i+1),(m-i-1)}$$
Then by Lemma \ref{lem:poly3}, $f_i(x^2)f_i(x^{-2}) = -z^2f_{(i+1)}(x^2)f_{(i+1)}(x^{-2})$ and the claim follows.

The coefficient of $z^{2k}$ in $f(x^2)f(x^{-2})$ is nonzero by assumption, and by Lemma \ref{lem:poly2} it is equal to
$$ (-1)^k \left( \alpha_{k,1} + \alpha_{k,2} + \ldots + \alpha_{k,(m-k)}\right)^2.$$
\end{proof}

\bibliographystyle{plain}
\bibliography{Timsbib1_2013.bib}
\end{document}